\documentclass [12pt,a4paper,reqno]{amsart}
\usepackage{amsmath,amsthm, amssymb,amscd,verbatim,delarray,verbatim}
\usepackage{amsfonts}

\catcode`\@=11

 \def\AMSTeXfeatures{\Plainheads 
   \let\current@vert=\AMS@vert}

 \def\Plainheads{\sh@ftdiam=0.05em
   \getlabeldims
   \let\vshaftfill=\plnvsolidfill
   \let\hshaftfill=\plnhsolidfill
   \let\th@rhead=\plnrhead
   \let\th@lhead=\plnlhead
   \let\th@dnhead=\plndnhead
   \let\th@uphead=\plnuphead}
 
 \def\glet{\global\let}

 \def\LaTeXfeatures{\catcode`\@=11
   \ifx\@clnwd\undefined \nol@g
      \input ltxcode.tex \dol@g \fi
   \ltxheads \let\current@vert=\new@vert
   \providelto \catcode`\@=\active}

 \def\nol@g{\def\wlog{\edef\garbage}}
 \def\dol@g{\let\wlog=\wl@g} \let\wl@g=\wlog
 \nol@g 

 \newbox\ltobox
 \def\providelto{{\setbox\z@=
   \hbox{$\to$}\minharrlen=\wd\z@
   \global\setbox\ltobox=\hbox{$\activeat>>>$}}
   \def\lto{\mathrel{\copy\ltobox}}}

 \def\ltxheads{\sh@ftdiam=\@wholewidth
   \getlabeldims
   \let\vshaftfill= \ltxvsolidfill
   \let\hshaftfill=\ltxhsolidfill
   \let\th@rhead=\ltxrhead
   \let\th@lhead=\ltxlhead
   \let\th@dnhead=\ltxdnhead
   \let\th@uphead=\ltxuphead}
 {\catcode`\@=\active
   \gdef@#1{\csname #1\string@at\endcsname}
   \glet\activeat=@}
 \def\def@#1{\expandafter\def\csname #1@at\endcsname}

 \def@>#1>#2>{\@rrow R{#1}{#2}}
 \def@<#1<#2<{\@rrow L{#1}{#2}}
 \def@ V#1V#2V{\@rrow V{#1}{#2}}
 \def@ A#1A#2A{\@rrow A{#1}{#2}}
 \def@/#1/#2/#3/{\@rrow{#1}{#2}{#3}}
 \def@.{\ifodd\row\ifmmode\noharrow
     \else\leavevmode.\spacefactor3000 \fi
   \else\novarrow\fi}
 \def@={\ifodd\row\harrow\hequalfill{}{}%
   \else\varrow\vequalfill{}{}\fi}
 \def@:#1{\ifx=#1\harrow\deffill{}{}%
   \else\leavevmode\null:#1\fi}
 \def@|{\current@vert}
  \def\AMS@vert{\varrow\vequalfill{}{}}
  \def\new@vert#1|#2|{\ifodd\row
   \let\nextarrow\vertexvarrow
   \else\let\nextarrow\varrow\fi
   \nextarrow\vshaftfill{#1}{#2}}
 \def@-{\ifmmode\let\next\hl@ne
   \else\let\next\AMSatdash \fi \next}
  \def\hl@ne#1-#2-{\harrow\hshaftfill{#1}{#2}}
  \def\AMSatdash{\let\next\relax\leavevmode
    \def\next@{\ifx\next-%
      \def\next-{\futurelet\next\nextii@}%
     \else\def\next{\hbox{-}}\fi\next}%
    \def\nextii@{\ifx\next-\def\next-{\hbox{---}}%
      \else\def\next{\hbox{--}}\fi\next}%
    \futurelet\next\next@}
 \def@(#1){\tweenarrows{#1}}
 \def@[#1]{\setsp@n#1\relax\activeat}
 \def\fiberbox{\hbox{$\vcenter{\hr@le\hbox{\vr@le
   \kern1ex\vbox{\kern1.2ex}\vr@le}\hr@le}$}}
  \def\hr@le{\hrule height \sh@ftdiam}
  \def\vr@le{\vrule width \sh@ftdiam}
 \def@+#1+#2+#3+{\ifodd\row \harrow{#1}{#2}{#3}%
   \else \varrow{#1}{#2}{#3}\fi}

 \def\Dnarrfill{\vequalfill\Dnhe@d}
 \def\Uparrfill{\Uphe@d\vequalfill}
 
 \def\ontofill{\rtarrfill\kern-0.3em 
   \th@rhead\kern 0.3em} 

 \def\rtarrfill{\hshaftfill\th@rhead}
 \def\ltarrfill{\th@lhead\hshaftfill}
 \def\dnarrfill{\vshaftfill\th@dnhead}
 \def\uparrfill{\th@uphead\vshaftfill}
 \def\hequalfill{\plnhfill=}
 \def\deffill{:\plnhfill=}
 \def\plnvextfill#1{\setbox\z@
   \hbox{\the\textfont3 #1}%
   \dimen@=\dp\z@\advance\dimen@\ht\z@
   \copy\z@ \kern-\dimen@ 
   \cleaders\copy\z@ \vfill
   \kern-\dimen@ 
   \box\z@}
 \def\plnhfill#1{$\m@th\mkern-1.5mu\mathord#1\mkern-6mu
    \cleaders\hbox{$\mkern-2mu\mathord#1\mkern-2mu$}\hfill
    \mkern-6mu\mathord#1\mkern-1.5mu$}
 \def\vequalfill{\plnvextfill{\char'167}}
 \def\plnvsolidfill{\plnvextfill{\char'077}}
 \def\plnhsolidfill{\plnhfill-}
 \def\ltxhsolidfill{\leaders\hrule height\topofshaft depth\botofshaft
   \hfill}
 \def\ltxvsolidfill{\leaders\vrule width\sh@ftdiam\vfill}
 \def\hdashfill{\hd@sh\wd@sh
   \xleaders \hbox{\wd@sh\hd@sh\wd@sh}\hfill
   \wd@sh\hd@sh}
 \def\vdashfill{\vd@sh\wd@sh
   \xleaders \vbox{\wd@sh\vd@sh\wd@sh}\vfill
   \wd@sh\vd@sh}
 \def\dashed{\ifinmeasureCD\else
    \ifodd\row\option{\let\hshaftfill=\hdashfill}%
   \else\option{\let\vshaftfill=\vdashfill}\fi\fi}

 \newdimen\CDstrutht  \newdimen\CDstrutdp
 \newdimen\CDstrutlen \CDstrutlen=\CDstrutht
   \advance\CDstrutlen by \CDstrutdp

 \def\CDstrut{\vrule
   height \ifnum\row=1 \z@\else\CDstrutht \fi
   depth \ifnum\row=\numrows \z@ \else\CDstrutdp \fi
   width\z@}

 \newdimen\CDarrsurr \CDarrsurr=0.375em
 \newdimen\CDdashlen
 \newdimen\CDvarrlen \CDvarrlen=1.5\baselineskip
 \newdimen\minharrlen 
  \setbox\z@\hbox{$\longrightarrow$} \minharrlen=\wd\z@
 \newdimen\minCDharrlen \minCDharrlen=2.5em 
\newdimen \minc@lwd
\def\findminc@lwd{\minc@lwd=2\CDarrsurr
  \advance\minc@lwd\minCDharrlen}

 \newdimen\sh@ftdiam

 \newdimen\labelsurr \labelsurr=1.25 em

\newcount\sp@ncnt \sp@ncnt=\@ne
\newcount\sp@ncnt@ \sp@ncnt@=\@ne
\newdimen\@rrwd \newdimen\@rrdp

 \def\adjustbot#1{\option{\advance\@rrdp#1\relax}}
 
\def\pushvertex#1{\global\p@shlen#1\relax
   \global\let\maybepush=\dopush}

 \newdimen\p@shlen \p@shlen=\z@

 \let\maybepush=\relax
 \def\dopush{\ifinmeasureCD 
   \advance\locdimen by -\p@shlen 
   \else\advance \@rrwd by -\p@shlen \fi 
   \global\let\maybepush=\relax \global\p@shlen=\z@\relax}

 \def\span@ne{\global\sp@ncnt=\@ne\relax}
 \def\setsp@n#1#2{\global\sp@ncnt=#1\relax
   \ifx\relax#2\relax\else\global\sp@ncnt@=#2\relax\fi}

 \def\plnrhead{\llap{$\rightarrow\mkern-1.5mu$}}
 \def\plnlhead{\rlap{$\mkern-1.5mu\leftarrow$}}

 \def\clap#1{\hbox to \z@{\hss #1\hss}}

 \def\plndnhead{\hbox{\the\textfont3 \char'171}}
 \def\plnuphead{\hbox{\the\textfont3 \char'170}}
 \def\Dnhe@d{\hbox{\the\textfont3 \char'177}}
 \def\Uphe@d{\hbox{\the\textfont3 \char'176}}

 \def\ltxrhead{\raise\@xisheight
   \llap{\smash{\@linefnt\@getrarrow(1,0)}}}
 \def\ltxlhead{\raise\@xisheight
   \rlap{\@linefnt\@getlarrow(-1,0)}}
 \def\ltxuphead{\setbox\z@=\rlap{%
   \kern\@halfwidth\@linefnt\char'66}%
   \copy\z@\kern-\ht\z@}
 \def\ltxdnhead{\setbox\z@=\rlap{%
   \kern\@halfwidth\@linefnt\char'77}%
   \ht\z@=\z@\box\z@}

 \def\wd@sh{\kern0.5\CDdashlen}
 \def\hd@sh{\vrule height\topofshaft depth\botofshaft
    width\CDdashlen}
 \def\vd@sh{\hrule height\CDdashlen
   depth\z@ width\sh@ftdiam}

\def\xylist{14{3434}13{2414}12{1723}%
  23{1413}34{1153}11{0867}43{0707}%
  32{0580}21{0414}31{0291}41{0}}
\newcount\tgtcnt@
\def\find@xyargs{\dimen@=\@rrdp
  \advance\dimen@ by \CDstrutlen
  \tgtcnt@=\dimen@ \dimen@=\@rrwd 
  \divide\dimen@ by \@m 
  \divide \tgtcnt@ by \dimen@ 
  \expandafter\testxy\xylist\relax
  \unitlength=\@xarg\@rrdp
  \divide\unitlength by\@yarg\relax}
\def\testxy#1#2#3{\ifnum\tgtcnt@>#3
    \@xarg=#1\relax \@yarg=#2\relax
    \let\next=\ignorerest
  \else\let\next\testxy\fi\next}
\def\ignorerest#1\relax{\relax}

\let\scalefactor=\@ne
\def\SWarrow{\find@xyargs\vector
  (-\@xarg,-\@yarg)\scalefactor\hskip-\wd\@linechar}
\def\NWarrow{\find@xyargs\vector
  (-\@xarg,\@yarg)\scalefactor\hskip-\wd\@linechar}
\def\NEarrow{\find@xyargs\vector
  (\@xarg,\@yarg)\scalefactor}
\def\SEarrow{\find@xyargs\vector
  (\@xarg,-\@yarg)\scalefactor}
\def\rightupline{\find@xyargs\@linelen=\scalefactor
     \unitlength\@sline}
\def\rightdownline{\find@xyargs\@yarg=-\@yarg\relax
     \@linelen=\scalefactor\unitlength\@sline}

\def\Sim{\ifodd\row\setbox\z@=\hbox{$\sim$}\dimen@=\ht\z@
 \advance\dimen@ by -\@xisheight
  \vbox{\box\z@\kern-\@xisheight\kern\dimen@}%
  \else\hbox{$\wr$}\fi}

\def\harrow#1#2#3{\inmeasureCDtrue\findminarrwd
  {#2}{#3}{\sp@ncnt\minharrlen}\inmeasureCDfalse\span@ne
  \mathrel{\hbox{\options\hplace{#1}\ulabel{#2}\dlabel{#3}}}}

\def\noharrow{\harrow\hfill{}{}}
\def\vertexvarrow#1#2#3{\findarrdp \@rrwd=\z@ \setsp@n\@ne\@ne
  \vbox to \z@{\kern-1.2\CDstrutht
  \rlap{\options\vplace{#1}\llabel{#2}\rlabel{#3}}\vss}}

\newif\ifinmeasureCD
\def\measurelabel#1{\setbox\z@
  \hbox{$\scriptstyle#1\kern\labelsurr$}%
  \ifdim\wd\z@>\@rrwd \@rrwd=\wd\z@\fi}
\def\findminarrwd#1#2#3{\@rrwd=#3\relax
   \measurelabel{#1}\measurelabel{#2}}
\def\findCDarrwd#1#2{\@rrwd=\minCDharrlen
   \measurelabel{#1}\measurelabel{#2}%
  }

\newcount\row \row=\@ne \newcount\col \col=\@ne 
 \newcount\numrows 
\numrows=\@ne
 \newcount\numcols
\newcount\arrspan \newdimen\vrtxhalfwd  \newbox\tempbox

\def\DANABUG{\advance\col by \@ne
 \@rrwd=\minCDharrlen
  \advance\@rrwd by \vrtxhalfwd
  \advance\@rrwd by \CDarrsurr
  \ifnum\col>\numcols \numcols=\col
     \newlocdimen{col\the\col}\locdimen=\@rrwd 
  \else \ifdim\@rrwd>\c@l \c@l=\@rrwd\fi\fi}

\def\drop#1\\{
  \findvrtxhalfsum\DANABUG\advance\row by 2 \measureinit}

\def\measureinit{\col=\@ne \vrtxhalfwd=-\CDarrsurr\arrspan=\@ne\@rrwd=\z@
   \setbox\tempbox=\hbox\bgroup$}
\def\measure{
  \let\harrow\measureCDarrow
  \let\CDCR=\measureCR 
   \findminc@lwd 
  \inmeasureCDtrue
  \row=\@ne \numcols=\z@ \measureinit}

\def\endmeasure{\findvrtxhalfsum\DANABUG
  \numrows=\row 
  \inmeasureCDfalse}

\def\newlocdimen#1{\advance\dimenc@unt by \@ne
  \ifnum\dimenc@unt<\insc@unt
     \else\errmessage{No room for the CD}\fi
  \dimendef\locdimen=\dimenc@unt
  \expandafter\dimendef\csname#1\endcsname=\dimenc@unt}

 \def\r@wc@l{\csname row\the\row col\the\col\endcsname}
 \def\c@l{\csname col\the\col\endcsname}

 \def\findvrtxhalfsum{$\egroup
  \newlocdimen{row\the\row col\the\col}
  \locdimen=\vrtxhalfwd 
  \vrtxhalfwd=0.5\wd\tempbox 
  \advance\vrtxhalfwd by \CDarrsurr
  \advance\locdimen by \vrtxhalfwd 
  \advance\@rrwd by \locdimen 
  \maybepush
  \divide\@rrwd by \arrspan\relax
  \ifdim\@rrwd<\minc@lwd
    \ifnum\col>\@ne \@rrwd=\minc@lwd\fi \fi
  \loop 
    \ifnum\col>\numcols \numcols=\col
       \newlocdimen{col\the\col}
       \locdimen=\@rrwd 
    \else \ifdim\@rrwd>\c@l \c@l=\@rrwd\fi \fi
   \ifnum\arrspan>\@ne
      \advance\arrspan by -1 \advance\col by \@ne
  \repeat }

 \def\measureCDarrow#1#2#3{\findvrtxhalfsum
   \arrspan=\sp@ncnt\relax\global\sp@ncnt=1\relax
   \advance\col by \@ne
   \findCDarrwd{#2}{#3}%
   \setbox\tempbox=\hbox\bgroup$}

 \newcount\dr@tn \dr@tn=\z@
 \def\locate#1:#2{\ifinmeasureCD\else
   \count@=-#1
   \multiply\count@ by 2
   \advance\count@ by #2
   \dimen@=\count@\@rrwd
   \ifnum\dr@tn=\@ne\relax \else\dimen@=-\dimen@ \fi
   \dimen@i=\@rrdp
   \ifnum\dr@tn>\z@\advance\dimen@i by \CDstrutlen \fi
   \dimen@i=\count@\dimen@i
   \count@=#2 \multiply\count@ by 2
   \divide\dimen@ by \count@
   \divide\dimen@i by \count@
   \lift\dimen@i\nudge\dimen@\fi}

\def\betweenCDrows{\advance\row by \@ne \col=\@ne
\options}

\def\hbegin{\hbox\bgroup\kern\c@l \kern-\r@wc@l$}
\def\hend{$\glet\maybepush\relax \CDstrut\egroup}
\def\vbegin{\setbox\tempbox=\hbox\bgroup$}
\def\vend{$\egroup\ht\tempbox=\z@\dp\tempbox\CDvarrlen
  \box\tempbox}
\def\setCD{\let\harrow=\setCDarrow
  \let\CDCR=\setCR 
  \row=\@ne \col=\@ne \hbegin}
\let\endsetCD=\hend 

\def\findarrwd{\@rrwd=\z@ \count@=\col \advance\count@ by\sp@ncnt
  \loop\ifnum\count@>\col \advance\count@ by -1
      \advance\@rrwd by\csname col\the\count@\endcsname\repeat}
\def\setCDarrow#1#2#3{\kern\CDarrsurr\advance\col by \@ne
  \findarrwd \advance\@rrwd by -\r@wc@l  
  \@rrdp=\z@ 
  \maybepush
  \advance\col by -\@ne \advance\col by \sp@ncnt \span@ne
  \hbox to \@rrwd{\options
   \@rrwd=\scalefactor\@rrwd\hss
   \hplace{#1}\ulabel{#2}\dlabel{#3}\hss}%
   \kern\CDarrsurr}

\newdimen\labspacei 
\newdimen\labspaceii 

\newdimen\@xisheight
  \@xisheight=\the\fontdimen22\textfont2
\newdimen\labelskip
  \labelskip=\the\fontdimen10\textfont3 
\newdimen\topofshaft
\newdimen\botofshaft
\newdimen\botofulabel
\newdimen\topofdlabel
\def\getlabeldims{
  \topofshaft=0.5\sh@ftdiam
  \botofshaft=\topofshaft
  \advance\topofshaft by \@xisheight  
  \advance\botofshaft by -\@xisheight  
  \botofulabel=\topofshaft
  \advance\botofulabel by \labelskip
  \topofdlabel=\botofshaft
  \advance\topofdlabel by \labelskip}

\def\ulabel{\ifnum\row=\@ne\let\next\ulabeli
   \else\let\next\ulabellap\fi\next}
\def\ulabeli#1{\vbox{
  \clap{\kern-\@rrwd$\scriptstyle#1$}%
  \kern\botofulabel}\maybeoffset}
\def\ulabellap#1{\vbox to \z@{\vss
  \clap{\kern-\@rrwd$\scriptstyle#1$}%
  \kern\botofulabel}\maybeoffset}
\def\dlabel{\ifnum\row=\numrows\let\next\dlabeli
   \else\let\next\dlabellap\fi\next}
\def\dlabeli#1{\vtop{\kern\topofdlabel
  \clap{\kern-\@rrwd$\scriptstyle#1$}%
  }\maybeoffset}
\def\dlabellap#1{\vbox to \z@{\kern\topofdlabel
  \clap{\kern-\@rrwd$\scriptstyle#1$}%
  \vss}\maybeoffset}
\def\rlabel#1{\vbox to \z@{\vss
  \rlap{\kern\labelskip$\scriptstyle#1$}%
  \vss\kern-\@rrdp}\maybeoffset}
\def\llabel#1{\vbox to \z@{\vss
  \llap{$\scriptstyle#1$\kern\labelskip}%
  \vss\kern-\@rrdp}\maybeoffset}
\def\swlabel#1{\vtop{\kern0.5\@rrdp
  \llap{$\scriptstyle#1$\kern\labelskip\kern-0.5\@rrwd}
  }\maybeoffset}
\def\nwlabel#1{\vbox{
  \llap{$\scriptstyle#1$\kern\labelskip\kern-0.5\@rrwd}%
  \kern-0.5\@rrdp}\maybeoffset}
\def\selabel#1{\vtop{\kern0.5\@rrdp
  \rlap{\kern0.5\@rrwd\kern\labelskip$\scriptstyle#1$}%
  }\maybeoffset}
\def\nelabel#1{\vbox{
  \rlap{\kern0.5\@rrwd\kern\labelskip$\scriptstyle#1$}%
  \kern-0.5\@rrdp}\maybeoffset}
\def\cplace#1{\vbox to \z@{\vss
  \clap{$#1$\kern-\@rrwd}%
  \kern-\@rrdp\vss}\maybeoffset}
\def\hplace#1{\hbox to \@rrwd{#1}\maybeoffset}
\def\vplace#1{\clap{\vbox to \z@{#1\kern-\@rrdp}}\maybeoffset}

\newdimen\nudgeamount \nudgeamount=\z@
\newdimen\liftamount \liftamount=\z@
\let\maybeoffset\relax
\newbox\offsetbox \newdimen\lastheight
\def\dooffset{
  \setbox\offsetbox=\lastbox \lastheight=\ht\offsetbox 
  \setbox\offsetbox=\vbox{\kern-\liftamount\box\offsetbox}%
  \ht\offsetbox=\lastheight
  \kern\nudgeamount\box\offsetbox\kern-\nudgeamount
  \global\nudgeamount=\z@ \global\liftamount=\z@
  \glet\maybeoffset=\relax}
\def\nudge#1{\ifinmeasureCD\else
  \global\advance\nudgeamount#1\relax
  \global\let\maybeoffset\dooffset\fi}
\def\lift#1{\ifinmeasureCD\else
  \global\advance\liftamount#1\relax
  \global\let\maybeoffset\dooffset\fi}

\def\findarrdp{\@rrdp=\CDvarrlen
  \ifnum\sp@ncnt@>1
    \advance\@rrdp by \CDstrutlen
    \multiply\@rrdp by \sp@ncnt@
    \advance\@rrdp by -\CDstrutlen \fi
 }

\def\varrow#1#2#3{\ifnum\sp@ncnt>\@ne 
     \sp@ncnt@=\sp@ncnt\relax\fi
  \findarrdp \@rrwd=\z@ 
  \kern\c@l
   \hbox to \z@{\options
   \@rrdp=\scalefactor\@rrdp
    \hss\vplace{#1}\llabel{#2}\rlabel{#3}\hss}%
  \global\advance\col by \@ne \setsp@n\@ne\@ne
  }

\def\novarrow{\varrow\vfill{}{}}

\def\tweenarrows#1{\findarrwd \findarrdp \setsp@n\@ne\@ne
  \rlap{\options\cplace{#1}}}

\def\usarrow #1#2#3{\dr@tn=\@ne
  \findarrwd \findarrdp \setsp@n\@ne\@ne 
  \rlap{\options\cplace{#1}\nwlabel{#2}\selabel{#3}}%
  \dr@tn=\z@}
\def\dsarrow #1#2#3{\dr@tn=\tw@
  \findarrwd \findarrdp \setsp@n\@ne\@ne 
  \rlap{\options\cplace{#1}\swlabel{#2}\nelabel{#3}}%
  \dr@tn=\z@}
 \def\@rrow#1{\csname #1@rrow\endcsname}
 \def\R@rrow{\harrow \rtarrfill}
 \def\L@rrow{\harrow \ltarrfill}
 \def\V@rrow{\varrow \dnarrfill}
 \def\A@rrow{\varrow \uparrfill}
 \def\SE@rrow{\dsarrow \SEarrow}
 \def\NW@rrow{\dsarrow \NWarrow}
 \def\SW@rrow{\usarrow \SWarrow}
 \def\NE@rrow{\usarrow \NEarrow}
 \def\DS@rrow{\dsarrow \dnslope}
 \def\US@rrow{\usarrow \upslope}
 \def\upslope{\find@xyargs
       \@linelen=\unitlength\@sline}
 \def\dnslope{\find@xyargs\@yarg=-\@yarg\relax
       \@linelen=\unitlength\@sline}

\newtoks\optionlist 
\optionlist={}
\let\options\relax
\def\dooptions{\the\optionlist\global\optionlist={}%
  \glet\options=\relax}
\def\option#1{\ifinmeasureCD\else
  \glet\options=\dooptions
  \global\optionlist=\expandafter{\the\optionlist\relax#1}\fi}
\def\wider#1{\ifinmeasureCD\else
   \option{\advance\@rrwd by #1}\fi}
\def\deeper#1{\ifinmeasureCD\else
   \option{\advance\@rrdp by #1}\fi}

{\def\\{\global\let\sptoken= }\\ }

\def\CR{\futurelet\nexttok\testCR}
\def\testCR{\ifx\nexttok\sptoken
   \let\next\eatspaceCR\else\let\next\CDCR\fi\next}
\def\eatspaceCR#1 {\CR}
\def\measureCR{\ifx\nexttok\endmeasure\let\nextCR\relax
    \else\let\nextCR\drop\fi\nextCR}
\def\setCR{\ifodd\row
  \ifx\nexttok\endsetCD\else\hend\betweenCDrows\vbegin\fi
  \else\vend\betweenCDrows\hbegin\fi}

\countdef\dimenc@unt=11
\def\CD#1\endCD{
   \begingroup\let\\=\CR
  \m@th\offinterlineskip
   \measure#1\endmeasure\null\,\vcenter{\setCD#1\endsetCD}\,
   \endgroup
    }

\ifx\@clnwd\undefined \nol@g\else\catcode`\ =14\relax\fi
 \font\@linefnt=line10 
 \newcount\@tempcnta
 \newcount\@tempcntb
 \newdimen\@tempdima
 \newdimen\@tempdimb
 \newdimen\@wholewidth
 \newdimen\@halfwidth
   \@wholewidth\fontdimen8\@linefnt \@halfwidth .5\@wholewidth
 \newdimen\unitlength
 \newcount\@xarg
 \newcount\@yarg
 \newcount\@yyarg
 \newbox\@linechar
 \newdimen\@linelen
 \newdimen\@clnwd
 \newdimen\@clnht
 \newif\if@negarg
 
 \def\@whilenoop#1{}

 \def\@whiledim#1\do #2{\ifdim #1\relax#2\@iwhiledim{#1\relax#2}\fi}

 \def\@iwhiledim#1{\ifdim #1\let\@nextwhile=\@iwhiledim 
         \else\let\@nextwhile=\@whilenoop\fi\@nextwhile{#1}}

 \def\@sline{\ifnum\@xarg< 0 \@negargtrue \@xarg -\@xarg \@yyarg -\@yarg
   \else \@negargfalse \@yyarg \@yarg \fi
 \ifnum \@yyarg >0 \@tempcnta\@yyarg \else \@tempcnta -\@yyarg \fi
 \ifnum\@tempcnta>6 \@badlinearg\@tempcnta0 \fi
 \ifnum\@xarg>6 \@badlinearg\@xarg 1 \fi
 \setbox\@linechar\hbox{\@linefnt\@getlinechar(\@xarg,\@yyarg)}%
 \ifnum \@yarg >0 \let\@upordown\raise \@clnht\z@
    \else\let\@upordown\lower \@clnht \ht\@linechar\fi
 \@clnwd=\wd\@linechar
 \if@negarg \hskip -\wd\@linechar \def\@tempa{\hskip -2\wd\@linechar}\else
      \let\@tempa\relax \fi
 \@whiledim \@clnwd <\@linelen \do
   {\@upordown\@clnht\copy\@linechar
    \@tempa
    \advance\@clnht \ht\@linechar
    \advance\@clnwd \wd\@linechar}%
 \advance\@clnht -\ht\@linechar
 \advance\@clnwd -\wd\@linechar
 \@tempdima\@linelen\advance\@tempdima -\@clnwd
 \@tempdimb\@tempdima\advance\@tempdimb -\wd\@linechar
 \if@negarg \hskip -\@tempdimb \else \hskip \@tempdimb \fi
 \multiply\@tempdima \@m
 \@tempcnta \@tempdima \@tempdima \wd\@linechar \divide\@tempcnta \@tempdima
 \@tempdima \ht\@linechar \multiply\@tempdima \@tempcnta
 \divide\@tempdima \@m
 \advance\@clnht \@tempdima
 \ifdim \@linelen <\wd\@linechar
    \hskip \wd\@linechar
   \else\@upordown\@clnht\copy\@linechar\fi}
 
 \def\@getlinechar(#1,#2){\@tempcnta#1\relax\multiply\@tempcnta 8
 \advance\@tempcnta -9 \ifnum #2>0 \advance\@tempcnta #2\relax\else
 \advance\@tempcnta -#2\relax\advance\@tempcnta 64 \fi
 \char\@tempcnta}
 
 \def\vector(#1,#2)#3{\@xarg #1\relax \@yarg #2\relax
 \@tempcnta \ifnum\@xarg<0 -\@xarg\else\@xarg\fi
 \ifnum\@tempcnta<5\relax
 \@linelen=#3\unitlength
 \ifnum\@xarg =0 \@vvector 
   \else \ifnum\@yarg =0 \@hvector \else \@svector\fi
 \fi
 \else\@badlinearg\fi}
 
 \def\@svector{\@sline
 \@tempcnta\@yarg \ifnum\@tempcnta <0 \@tempcnta=-\@tempcnta\fi
 \ifnum\@tempcnta <5
   \hskip -\wd\@linechar
   \@upordown\@clnht \hbox{\@linefnt  \if@negarg 
   \@getlarrow(\@xarg,\@yyarg) \else \@getrarrow(\@xarg,\@yyarg) \fi}%
 \else\@badlinearg\fi}
 
 \def\@getlarrow(#1,#2){\ifnum #2 =\z@ \@tempcnta='33\else
 \@tempcnta=#1\relax\multiply\@tempcnta \sixt@@n \advance\@tempcnta
 -9 \@tempcntb=#2\relax\multiply\@tempcntb \tw@
 \ifnum \@tempcntb >0 \advance\@tempcnta \@tempcntb\relax
 \else\advance\@tempcnta -\@tempcntb\advance\@tempcnta 64
 \fi\fi\char\@tempcnta}
 
 \def\@getrarrow(#1,#2){\@tempcntb=#2\relax
 \ifnum\@tempcntb < 0 \@tempcntb=-\@tempcntb\relax\fi
 \ifcase \@tempcntb\relax \@tempcnta='55 \or 
 \ifnum #1<3 \@tempcnta=#1\relax\multiply\@tempcnta
 24 \advance\@tempcnta -6 \else \ifnum #1=3 \@tempcnta=49
 \else\@tempcnta=58 \fi\fi\or 
 \ifnum #1<3 \@tempcnta=#1\relax\multiply\@tempcnta
 24 \advance\@tempcnta -3 \else \@tempcnta=51\fi\or 
 \@tempcnta=#1\relax\multiply\@tempcnta
 \sixt@@n \advance\@tempcnta -\tw@ \else
 \@tempcnta=#1\relax\multiply\@tempcnta
 \sixt@@n \advance\@tempcnta 7 \fi\ifnum #2<0 \advance\@tempcnta 64 \fi
 \char\@tempcnta}
\catcode`\ =10

\dol@g 
\catcode`\@=\active
\LaTeXfeatures

 \input amssymb.sty
\newcommand{\F}{{\mathbb F}}

\newcommand{\Q}{{\mathbb Q}}

\numberwithin{equation}{section}

\DeclareMathOperator{\Mod}{mod}

\begin{document}

\newtheorem{theorem}{Theorem}[section]
\newtheorem{lemma}[theorem]{Lemma}
\newtheorem{prop}[theorem]{Proposition}
\newtheorem{proposition}[theorem]{Proposition}
\newtheorem{corollary}[theorem]{Corollary}
\newtheorem{corol}[theorem]{Corollary}
\newtheorem{conj}[theorem]{Conjecture}

\theoremstyle{definition}
\newtheorem{defn}[theorem]{Definition}
\newtheorem{example}[theorem]{Example}
\newtheorem{examples}[theorem]{Examples}
\newtheorem{remarks}[theorem]{Remarks}
\newtheorem{remark}[theorem]{Remark}
\newtheorem{algorithm}[theorem]{Algorithm}
\newtheorem{question}[theorem]{Question}
\newtheorem{problem}[theorem]{Problem}
\newtheorem{subsec}[theorem]{}
\newtheorem{acknowledgements}[theorem]{Acknowledgements \nonumber}

\def\toeq{{\stackrel{\sim}{\longrightarrow}}}
\def\into{{\hookrightarrow}}


\def\alp{{\alpha}}  \def\bet{{\beta}} \def\gam{{\gamma}}
 \def\del{{\delta}}
\def\eps{{\varepsilon}}
\def\kap{{\kappa}}                   \def\Chi{\text{X}}
\def\lam{{\lambda}}
 \def\sig{{\sigma}}  \def\vphi{{\varphi}} \def\om{{\omega}}
\def\Gam{{\Gamma}}   \def\Del{{\Delta}}
\def\Sig{{\Sigma}}   \def\Om{{\Omega}}
\def\ups{{\upsilon}}


\def\F{{\mathbb{F}}}
\def\BF{{\mathbb{F}}}
\def\BN{{\mathbb{N}}}
\def\Q{{\mathbb{Q}}}
\def\Ql{{\overline{\Q }_{\ell }}}
\def\CC{{\mathbb{C}}}
\def\R{{\mathbb R}}
\def\V{{\mathbf V}}
\def\D{{\mathbf D}}
\def\BZ{{\mathbb Z}}
\def\K{{\mathbf K}}
\def\XX{\mathbf{X}^*}
\def\xx{\mathbf{X}_*}

\def\AA{\Bbb A}
\def\BA{\mathbb A}
\def\HH{\mathbb H}
\def\PP{\Bbb P}

\def\Gm{{{\mathbb G}_{\textrm{m}}}}
\def\Gmk{{{\mathbb G}_{\textrm m,k}}}
\def\GmL{{\mathbb G_{{\textrm m},L}}}
\def\Ga{{{\mathbb G}_a}}

\def\Fb{{\overline{\F }}}
\def\Kb{{\overline K}}
\def\Yb{{\overline Y}}
\def\Xb{{\overline X}}
\def\Tb{{\overline T}}
\def\Bb{{\overline B}}
\def\Gb{{\bar{G}}}
\def\Ub{{\overline U}}
\def\Vb{{\overline V}}
\def\Hb{{\bar{H}}}
\def\kb{{\bar{k}}}

\def\Th{{\hat T}}
\def\Bh{{\hat B}}
\def\Gh{{\hat G}}

\def\cF{{\mathfrak{F}}}
\def\cC{{\mathcal C}}
\def\cU{{\mathcal U}}

\def\Xt{{\widetilde X}}
\def\Gt{{\widetilde G}}

\def\gg{{\mathfrak g}}
\def\hh{{\mathfrak h}}
\def\lie{\mathfrak a}

\def\GL{\textrm{GL}}            \def\Stab{\textrm{Stab}}
\def\Gal{\textrm{Gal}}          \def\Aut{\textrm{Aut\,}}
\def\Lie{\textrm{Lie\,}}        \def\Ext{\textrm{Ext}}
\def\PSL{\textrm{PSL}}          \def\SL{\textrm{SL}}
\def\loc{\textrm{loc}}
\def\coker{\textrm{coker\,}}    \def\Hom{\textrm{Hom}}
\def\im{\textrm{im\,}}           \def\int{\textrm{int}}
\def\inv{\textrm{inv}}           \def\can{\textrm{can}}
\def\Cl{\textrm{Cl}}
\def\Sz{\textrm{Sz}}
\def\ad{\textrm{ad\,}}
\def\SU{\textrm{SU}}
\def\PSL{\textrm{PSL}}
\def\PSU{\textrm{PSU}}
\def\rk{\textrm{rk}}
\def\PGL{\textrm{PGL}}
\def\Ker{\textrm{Ker}}
\def\Ob{\textrm{Ob}}
\def\Var{\textrm{Var}}
\def\poSet{\textrm{poSet}}
\def\Al{\textrm{Al}}
\def\Int{\textrm{Int}}
\def\Mod{\textrm{Mod}}
\def\Smg{\textrm{Smg}}
\def\ISmg{\textrm{ISmg}}
\def\Ass{\textrm{Ass}}
\def\Grp{\textrm{Grp}}
\def\Com{\textrm{Com}}
\def\Im{\textrm{Im}}
\def\Val{\textrm{Val}}
\def\LKer{\textrm{LKer}}
\def\Val{\textrm{Val}}
\def\Th{\textrm{Th}}
\def\Set{\textrm{Set}}
\def\Hal{\textrm{Hal}}
\def\Lat{\textrm{Lat}}
\def\LK{\textrm{LK}}

\def\tors{_\def{\textrm{tors}}}      \def\tor{^{\textrm{tor}}}
\def\red{^{\textrm{red}}}         \def\nt{^{\textrm{ssu}}}

\def\sss{^{\textrm{ss}}}          \def\uu{^{\textrm{u}}}
\def\mm{^{\textrm{m}}}
\def\tm{^\times}                  \def\mult{^{\textrm{mult}}}

\def\uss{^{\textrm{ssu}}}         \def\ssu{^{\textrm{ssu}}}
\def\comp{_{\textrm{c}}}
\def\ab{_{\textrm{ab}}}

\def\et{_{\textrm{\'et}}}
\def\nr{_{\textrm{nr}}}

\def\nil{_{\textrm{nil}}}
\def\sol{_{\textrm{sol}}}
\def\End{\textrm{End\,}}

\def\til{\;\widetilde{}\;}

\large



\title[Isotyped algebras]{Isotyped algebras }

\author[Boris   Plotkin]{Boris Plotkin}
\address{Department of Mathematics, Hebrew University, Jerusalem, Israel}
\email{plotkin@macs.biu.ac.il}


\begin{abstract}
The paper is essentially a continuation of \cite{PZ}, whose main
notion is that of logic-geometrical equivalence of algebras
(LG-equivalence of algebras). This equivalence of algebras is stronger than
elementary equivalence. In the paper we introduce the notion of isotyped algebras and relate it to
LG-equivalence. We show that these notions coincide. The
idea of the type is one of the central ideas in Model Theory. The correspondence introduced in the paper  stimulates
 a bunch of problems which connect  universal algebraic
geometry and Model Theory. 
 We provide a new general
view on the subject, arising "on the territory" of universal
algebraic geometry. 
This insight   yields also applications of algebraic logic
   in  Model Theory. Application of algebraic logic in Model theory makes some approaches more transparent.

\end{abstract}

\maketitle


CONTENT

1. General view.

2. Logical noetherianity

3. Isotypeness and isomorphism

4. Logically perfect algebras

5. Some facts from algebraic logic. Appendix

\section{General view} \label{sec:general_view}

\subsection{Introduction}\label{subsec:intro}
The notion of a type of an algebra, like of
any other algebraic system, came from Model Theory and turns to be
one of its key notions (see, for example, \cite{Ho}, \cite{Ma}).

 The paper is devoted to isotyped
algebras that is to the algebras with the same types. We consider algebras which belong to a fixed variety of algebras $\Theta$.
We approach to the
notion of a type from the positions of {\it universal algebraic geometry} (UAG).
On the one hand, universal algebraic geometry is an equational
algebraic geometry  in an arbitrary
variety of algebras $\Theta$. This means that algebraic sets are defined by systems of equations
in free algebras from $\Theta$. On the other hand, universal algebraic geometry spreads to  First Order
Logic (FOL)  geometry   in an arbitrary $\Theta$ ({\it logical geometry}).
This means that algebraic sets are defined  by arbitrary first order formulas,
semantically compressed in the given $\Theta$. In the case of logical geometry algebraic sets are called elementary sets and arbitrary first order formulas replace equations
(see \cite{PZ} for details).

 In the paper we proceed from the system of notions of algebraic
logic.  In principle, it is possible to translate this approach to the usual model theoretic language. However, we  believe that application of algebraic logic makes the main ideas of the paper more transparent and consistent.

The bridge between logical geometry and model theory is provided via algebraic logic by the means of the algebra of formulas $\Phi = \Phi(X)$, $X$ is a finite set of variables. In fact,  $\Phi(X)$ is {\it a set of first order formulas over $X$} which is converted in a special way into an algebra of formulas. The precise definition of the algebra $\Phi = \Phi(X)$ is given in Section \ref{sec:appendix}.5 (see also \cite{P6}-\cite{P8}).

Let $W=W(X)$ be the free algebra in $\Theta$
over $X$. An equality
$w\equiv w'$, which is an element in the algebra $\Phi(X)$,
corresponds to an equation $w= w'$ in $W(X)$. So, equalities are considered as nullary operations (constants) in $\Phi(X)$. Boolean algebras with equalities of the form $w\equiv w'$, and
 with acting quantifiers $\exists x$ by all $x \in X$ are called {\it extended boolean algebras} (for the list of identities see \cite{PZ}, Section 2.1  and Subsection 5.2 of this paper). The algebra  $\Phi(X)$ is an example of an extended boolean algebra. However, in  $\Phi(X)$ there are other operations $s_*$ (see below).


It was mentioned that universal algebraic geometry is an algebraic geometry associated to an arbitrary variety of algebras $\Theta$. If $\Theta=Com-P$ is the variety of commutative associative algebras with unit over the field $P$ then we arrive to classical algebraic geometry. One of the principal problems related to universal algebraic geometry is to understand what a part of rich geometry of the variety $\Theta=Com-P$ (i.e. of the classical algebraic geometry) survives in other varieties $\Theta$. On the other hand, the new ideas related to UAG  and, especially, to logical geometry  appear in the classical situation of $Com-P$.

  It is quite important to note 
  that equational algebraic geometry (AG for short) is connected with the category $\Theta^0$ of the free in $\Theta$ algebras $W=W(X)$ with finite sets $X$. 
  In order to easy the intuition  one should mention that in classical case $\Theta^0$ is just the category of all polynomial algebras over a field $P$. The role of $\Theta^0$ in logical geometry  (LG for short) plays the special category $Hal_\Theta^0$, whose objects are the algebras of formulas $\Phi(X)$ (we use the notation "$Hal$" in order to remind the role played by P.Halmos in algebraic logic).  The categories $\Theta^0$ and $Hal_\Theta^0$ are bounded by the covariant functor

  $$
  \Theta^0\to Hal_\Theta^0.
  $$

This functor attaches a morphism $s_*: \Phi(X)\to\Phi(Y)$ in $Hal_\Theta^0$ to each homomorphism-morphism $s:W(X)\to W(Y)$ in $\Theta^0$.  Here $s_{*}$ is a boolean homomorphism which is in some sense compatible with quantifiers and equalities. The same $s_*$
can be treated as an operation in  special multi-sorted Halmos algebras (see Section 5 for details). In particular, the operation $s_*$ can  present in  a record of  elements from $\Phi(X)$.

Now let us make one more step towards the general theory. Denote by $\Phi^0=\Phi^0(X)$ the subalgebra in $\Phi(X)$ generated by all equalities in the signature of the boolean operations and quantifiers.  This is an extended boolean algebra which is a subalgebra in the extended boolean algebra $\Phi(X)$. This subalgebra is usually less than $\Phi(X)$ since the operations $s_*$ are not involved in the records of the elements from $\Phi^0$. We shall note here, that the algebra $\Phi^0(X)$ cannot be defined independently from the algebra $\Phi(X)$. In its turn $\Phi(X)$ is defined through the means of algebraic logic.

\begin{remark}  Axioms of Halmos algebras (Subsection 5.3) imply that if we have $s:W(Y)\to W(X)$ and $v$ is an equality in $\Phi^0(Y)$  then $s_*(v)$ is an equality in $\Phi^0(X)$. However, in general, for the formula $v=\exists yv_0 \in \Phi^0(Y)$ the formula $s_*(v)$ can be not in $\Phi^0(X)$. (See, for example, Proposition \ref{prop:cycl}).
\end{remark}

Now we shall consider the origin of the morphisms and operations  $s_*$. The reasons for introducing $s_*$ are as follows. It is well-known that the standard algebraic first order logic uses an infinite number of variables. Denote the infinite set of variables by  $X^0$. Assuming the needs of logical geometry we shall deal with the system $\Gamma$ of all finite subsets $X$ in  $X^0$. In algebraic logic this approach leads to Halmos categories and multi-sorted (i.e., $\Gamma$-sorted) Halmos algebras. Different $\Phi(X)$, $X\in \Gamma$ should be somehow connected. This peculiarity requires introduction of the operations and morphisms of the type $s_*$ (see Section 5 for details).


 We use some notions which can be found in \cite{PZ}. For the sake of completeness many 
 definitions from \cite{PZ} are reproduced throughout the paper. Papers devoted to universal algebraic geometry (\cite{BMR},\cite{KMR}, \cite{MR},\cite{P6},\cite{P3},\cite{P7},\cite{P8}, etc., ) makes the material more friendly.

 The paper is organized as follows. In Section 1 we introduce the notion of isotyped algebras and relate it to the known notions of universal algebraic geometry. Section 2 deals with noetherian properties. In Section 3 we introduce the notion of a logically separable algebra and give examples of isotyped but not isomorphic algebras. Section 4 is devoted to logically perfect algebras. We show that every group can be embedded into a logically perfect one. We also consider two examples of abelian groups and study their behavior with respect to logical properties. Section 5 makes the paper self-complete. In this section we provide the reader with the notions of algebraic logic. In particular we define the algebra of formulas $\Phi(X)$ which plays a principal  role in all considerations. In Section 5 there is also a list of open problems.

 In this paper we have no difficult theorems. But we have some new insight and new problems in the field of mathematics which can be characterized as Logic in Algebra. 
 Concerning the objects of investigation, in this paper an algebra $H\in\Theta$ takes a primary role and is considered from the perspective of its logical and geometric invariants.

\subsection{Algebra $Bool(W(X),H)$}\label{subsec:Bool}

 Let $H$ be an algebra in $\Theta$. Take a free in
$\Theta$ algebra $W=W(X)$ with a finite $X$ and consider the set
$Hom(W,H)$ as an affine space. Its points are homomorphisms $\mu : W
\to H$. If $X=\{ x_1, \ldots , x_n \}$, then $Hom(W,H)$ is isomorphic to $H^{(n)}$ and  a point $\mu$ meets a
tuple $\bar a = (a_1, \ldots , a_n) \in H^{(n)}$.

Denote by $Bool(W(X),H)$ the boolean algebra of all subsets of
$Hom(W(X),H)$. Define the action of quantifiers $\exists x, x \in X$.
Recall  (see \cite{H}) that if $B$ is a boolean algebra, then the mapping $\exists :
B \to B$ is an existential quantifier if

1. $\exists (0) =0$,

2. $\exists (a) \geq a$,

3. $\exists (a \wedge \exists b) = \exists a \wedge \exists b$.

\noindent
A universal quantifier $\forall : B \to B$ is defined dually as
$\forall (a) = \neg (\exists (\neg a ))$.

Let now $A \in Bool(W(X),H)$ and $x \in X$. We set: $\mu \in \exists x
A$ if there is  $\nu \in A$ such that $\mu(x') =\nu(x')$ for every
$x'\in X$, $x'\neq x$. The necessary conditions $(1)-(3)$ hold true and the
definition of existential quantifier perfectly agrees with 
intuition.

Let, further, $w\equiv w'$ be an equality in the algebra of formulas $\Phi(X)$. Define the
corresponding element of the algebra $Bool(W(X),H)$ by
$Val^X_H(w\equiv w') = \{ \mu: W \to H | (w,w') \in Ker(\mu) \}$. These elements are considered as equalities in $Bool(W(X),H)$.

Thus  $Bool(W(X),H)$ is defined as an extended Boolean algebra.

 As
we will see in Subsection 5.6, the correspondence  $w\equiv w' \to Val^X_H(w\equiv w')$ is
naturally extended up to a homomorphism value of extended boolean algebras
$$Val^X_H : \Phi(X) \to
Bool(W(X),H). \eqno(*)$$

We consider also the category $Hal_\Theta(H)$ of all $Bool(W(X),H)$ with the natural morphisms $s_\ast$. The values and morphisms in $Hal_\Theta^0$ and $Hal_\Theta(H)$ are connected by the following commutative diagram

$$
\CD
\Phi(X)@> s_\ast>> \Phi(Y)\\
@V \Val^X_H VV @VV \Val^Y_H V\qquad\qquad\qquad(**)   \\
Bool(W(X),H) @>s_\ast>> Bool(W(Y),H).
\endCD
$$

\medskip

\begin{remark} The existence of homomorphisms $\Val^X_H$ for every algebra $H$ in $\Theta$ which satisfy the diagram above was a leading idea towards the definition of algebra of formulas $\Phi(X)$ and the category of such algebras of formulas $Hal_\Theta^0$. This is a place we faced with advantages of application of algebraic logic.

More precisely, our aim is to define the system of algebras of formulas $\Phi(X)$, where $X$ are finite subsets of $X^0$ and to define the category $Hal_\Theta^0$ of these algebras in such a way that  for every algebra $H\in\Theta$ and every $s:W(X)\to W(Y)$ the conditions $(*)$ and $(**)$ are fulfilled.

 In fact, what we have to do is to define the upper level of the diagram $(**)$ and the vertical arrows having already the lower level of the diagram.
This leads to constructions  and conditions which are  realized in Section 5.

Such approach to the definition of the algebra $\Phi(X)$ and the category  $Hal_\Theta^0$ is well coordinated with the approach based on application of the equivalence of Lindenbaum-Tarski.

\end{remark}

\subsection{Logical kernel of a point, types, isotyped algebras}

Define the  notion of the {\it logical kernel} of a point $\mu: W(X)
\to H$. So, along with $Ker(\mu)$ we will consider the logical kernel $LKer(\mu)$. This logical kernel is an important logical invariant of a point.

\begin{defn}
Let
$u \in \Phi(X)$ and $\mu: W(X)
\to H$ be a point in $Hom(W(X),H)$. We set: $u \in LKer(\mu)$, if $\mu \in Val^X_H
(u)$.
\end{defn}

In this case we say that $Val^X_H (u)$ is the value of a formula $u$ in
$Bool(W(X),H)$ and a point $\mu$ is a solution of the "equation" $u$
in $Hom(W(X),H)$.
This definition corresponds to the usual inductive definition of a point satisfying a formula (see \cite{Ma}).

We have also $$Ker(\mu) = LKer(\mu)\cap M_X,$$
where $M_X$ is a set of all equalities $w\equiv w', \ w,w' \in
W(X)$.

Show that the kernel $LKer(\mu)$ is an ultrafilter of the boolean
algebra $\Phi(X)$. First prove that it is a filter. Let $u_1, u_2
\in LKer(\mu)$. We have $\mu \in Val^X_H (u_1) \cap Val^X_H (u_2)=
Val^X_H (u_1 \wedge u_2)$. Hence, $u_1 \wedge u_2 \in LKer(\mu)$.
Let, now, $u \in LKer(\mu)$, $v \in \Phi(X)$. We have: $\mu \in
Val^X_H (u)\cup  Val^X_H (v) = Val^X_H (u \vee v)$. Thus, $u \vee v
\in LKer(\mu)$ and $LKer(\mu)$ is a filter.

Let now $u \in \Phi(X), u \not \in LKer(\mu)$, i.e., $ \mu \not \in
Val^X_H (u), \mu \in \neg Val^X_H (u) = Val^X_H (\neg u)$. Then
$\neg u \in LKer(\mu)$ and, hence, $LKer(\mu)$  is an ultrafilter.

\begin{defn}
Every ultrafilter in $\Phi(X)$ we call $X$-type. A type $T$ we call
$X$-type of the algebra $H \in \Theta$, if $T = LKer(\mu)$ for some
$\mu : W(X) \to H$.
\end{defn}

\begin{remark}
Compare with the definition of a type from \cite{Ma}.
\end{remark}

We say that $T$ is realized in the algebra $H$ if $T$ is an $X$-type of $H$.
Denote by $S^X(H)$ the system of all $X$-types of the algebra $H$. This is an important logical invariant of the algebra $H$.

\begin{defn}
Algebras $H_1$ and $H_2$ in $\Theta$ are called isotyped, if for any
finite $X$ every $X$-type of the algebra $H_1$ is an $X$-type of the
algebra $H_2$ and vice versa.
\end{defn}

Thus, the algebras $H_1$ and $H_2$ are {\it isotyped} if
$$S^X(H_1)=S^X(H_2).$$
\noindent
for any $X$.
Note, further, that the logical kernel $LKer(\mu)$ of $\mu:W(X)\to H$ contains the elementary $X$-theory of the algebra $H$. Indeed, if $u\in Th^X(H)$ then $Val^H_X(u)=Hom(W(X),H)$. In particular, $\mu\in Val^X_H(u)$ and $u\in LKer(\mu)$. Thus $Th^X(H)\subset LKer(\mu).$

It is clear now that if $H_1$ and $H_2$ are isotyped then
 $$Th^X(H_1)= Th^X(H_2),$$
  where $Th^X(H)$ is the {\it elementary $X$-theory} of $H$.

\begin{defn}
We say that an algebra $H$ is {\it saturated}, if for every $X$ any ultrafilter $T$ in $\Phi(X)$ which contains  $Th^X(H)$ is realizable in $H$.
\end{defn}

\begin{remark}
 Observe that the definition of a saturated algebra which is used in Model theory operates also with constants. In our definition constants are already incorporated in the signature of the variety $\Theta$.
 \end{remark}

Now we can state that if the saturated algebras $H_1$ and $H_2$ are elementary equivalent i.e., $Th(H_1)=Th(H_2)$  then they are isotyped. Thus, the saturated algebras are elementary equivalent if and only if they are isotyped. As we noticed above the isotyped algebras are always elementary equivalent. However, in general the notion to be isotyped is more strong than to be elementary equivalent. In fact the relation on algebras "to be isotyped" can be treated as a generalization of the idea of saturated algebras.

In what follows we will consider logically noetherian algebras. The statement for such algebras is  also two sided: logically noetherian algebras $H_1$ and $H_2$ are elementary equivalent if and only if they are isotyped.

We consider the definition of isotyped algebras from the logical
geometry perspective, but first we treat (equational) algebraic
geometry. Note that in algebraic geometry it is also useful to speak
of {\it atomic kernel}. It is all equalities $w\equiv w'$ and inequalities
$w\not \equiv w'$ belonging to $LKer(\mu)$. Denote this atomic
kernel by $AtKer(\mu)$. If $w\equiv w' \not \in AtKer(\mu)$, then
$w\not \equiv w'  \in AtKer(\mu)$. In fact, $AtKer(\mu)$ is the kernel of $\mu$ represented in the algebra of formulas $\Phi(X)$.

We consider also a special logical  kernel  $LKer^0(\mu)$ defined by
$$
LKer^0(\mu)=LKer(\mu)\cap\Phi^0(X).
$$
\noindent
This kernel is an ultrafilter in the algebra $\Phi^0(X)$.

\subsection{The main Galois
correspondences in algebraic geometry
and logical geometry. $AG$- and $LG$-equivalence of algebras}

Consider the Galois correspondence between subsets $A$ in $Hom(W(X),H)$
and systems of equations $T$ in $W=W(X)$. These $T$ can be viewed as
binary relations in $W$. For each $T$ we set:
$$T'_H = A = \{\mu :W \to H | T \subset Ker(\mu)\}.$$
For an arbitrary $A$ we set
$$ A'_H = T = \bigcap_{\mu \in A} Ker(\mu).$$
We have here the Galois correspondence and we can speak of the Galois
closures $A''_H$, $T''_H$. Every set $A\subset Hom(W(X),H) $ of the form $A=T'_H$ is closed, and we call it {\it an
algebraic set}. Every system of equations $T$ of the form $T=A'_H$ is an $H$-closed congruence in $W$.

Let us do the same for the logical geometry, substituting $Ker(\mu)$
by $LKer(\mu)$. Here $T$ is an arbitrary subset in $\Phi = \Phi(X)$.
We set:
$$T^L_H = A = \{\mu :W \to H | T \subset LKer(\mu)\},$$
$$ A^L_H = T = \bigcap_{\mu \in A} LKer(\mu).$$
The corresponding  closures are $T^{LL}_H$ and $A^{LL}_H$. Each $A$ of the form $A=T^L_H$ is called an
{\it elementary set} (it can be defined for an infinite $T$ as well).  Each
$T=A^L_H$ is an $H$-closed filter in the boolean algebra $\Phi =
\Phi(X)$. We have also
$$ T^L_H = A = \bigcap_{u \in T} Val^X_H (u),$$
$$ T = A^L_H = \{ u \in \Phi(X) | A \subset Val^X_H (u) \}.$$

Recall that algebras $H_1$ and $H_2$ in $\Theta$ are called
{\it geometrically (AG-equivalent)}, if for any finite $X$ and any $T$ in
$W(X)$ we have $$T''_{H_1}= T''_{H_2}$$ (see \cite{P6}, \cite{P7}), and $H_1$ and $H_2$ are called
{\it $LG$-equivalent} if always  $$T^{LL}_{H_1} = T^{LL}_{H_2}$$ for $T
\subset \Phi(X)$ (see \cite{PZ}).

Note that starting from the logical kernel $LKer^0(\mu)$ one can also establish the Galois correspondence between subsets $T$ in $\Phi^0(X)$ and the sets of points $A$ in $Hom(W(X),H).$

\subsection{$LG$-equivalence and isotyped algebras}\label{sub:lg}
We interrelate $LG$-equivalence and isotypeness of algebras. Let $A$
be a subset in $Hom(W(X),H)$, consisting of a single point $\mu:
W(X) \to H$. We have $$ T = A^L_H = \{\mu\}^L_H = LKer(\mu).$$
Hence, $LKer(\mu)= \{\mu\}^L_H$ and $LKer(\mu)$ is an $H$-closed
ultrafilter. Let $T$ be an ultrafilter, and $T=LKer(\mu)$. Take
$T^L_H = A_0$ and let $\nu \in A_0$. We have $\{\nu\}^L_H =
LKer(\nu) \supset {A_0}^L_H = T = LKer(\mu)$. We see that any two
points in $A_0$ have the same logical kernel. Besides, we can note
that if $T$ is an ultrafilter in $\Phi(X)$, then this $T$ is
$X$-type for $H$ if and only if the set $A_0=T^L_H$ is not empty.

\begin{defn}\label{defn:rho}
Define the equivalence $\rho = \rho^X_H$ on the set $Hom(W(X),H)$
setting $\mu \rho \nu$ if and only if $LKer(\mu) =LKer(\nu)$.
\end{defn}
Consider
the quotient set $Hom(W(X),H) / \rho =
\overline{Hom}(W(X),H)$. Hence, there is a bijection $\overline{Hom}(W(X),H) \to S^X(H)$.


Note also that every coset of the equivalence $\rho$ is an
elementary set, defined by the type $LKer(\mu) =T$, where $\mu$ belongs
to the coset. According to \cite{P7}, each elementary set is
invariant under the action of the group of automorphisms $Aut(H)$.
Hence, if $\mu$ and $\nu$ are conjugated by an automorphism $\sigma
\in Aut(H)$, then $\mu \rho \nu$. Recall that the action of a group
$Aut(H)$ in $Hom(W,H)$ is defined by a transition $\mu \to \mu
\sigma$. Under certain conditions the opposite is true as well. In
these cases cosets of the equivalence $\rho$ are exactly the orbits
of the action of the group $Aut(H)$.

The following theorem is the main one.

\begin{theorem}
Algebras $H_1$ and $H_2$ are $LG$-equivalent if and only if they are
isotyped.
\end{theorem}

\begin{proof} Let $H_1$ and $H_2$ be $LG$-equivalent algebras. This
means that for any finite $X$ and any set $T$ of formulas from $\Phi(X)$ this $T$ is
$H_1$-closed if and only if $T$ is $H_2$-closed.

Take now  an ultrafilter $T$ in $\Phi(X)$ and let $T$ be an $X$-type
over $H_1$. Such $T$ is $H_1$-closed and, consequently,
$H_2$-closed. Here $T^L_{H_2}$ is not empty and, hence, $T$ is an
$X$-type over $H_2$. The transition from $H_2$ to $H_1$ works in a similar way.
Therefore, $H_1$ and $H_2$ are isotyped.

Let, further, $H_1$ and $H_2$ be isotyped. This means, in
particular, that if $T=LKer(\mu)$ for $\mu:W(X) \to H_1$, then
$T=LKer(\nu)$ for some $\nu:W(X) \to H_2$ as well. The same is true for the
transition from $H_2$ to $H_1$. Thus $T$ is simultaneously $H_1-$ and
$H_2-$closed.

Let, now, $T$ be an $H_1$-closed filter in $\Phi(X)$. We want to check that $T$ is $H_2-$closed.  Take
$A=T^L_{H_1}$. Then $T=A^L_{H_1}= \bigcap_{\mu \in A}LKer(\mu)$.
This also means that $u \in T$ if and only if $A \subset
Val^X_{H_1}(u)$.

We will see that the intersection of $H$-closed filters is also
$H$-closed. As we know every $H_1-$closed filter of the form $LKer(\mu)$ is  $H_2-$closed.
Thus $T= \bigcap_{\mu \in A}LKer(\mu)$ is an  $H_2$-closed filter. Similarly,
$H_2$-closedness of $T$ implies its $H_1$-closedness. This is true
for any $X$. Hence, $H_1$ and $H_2$ are LG-equivalent.

It remains to check that the intersection $T=\bigcap_\alpha
T_\alpha$ with all $H$-closed $T_\alpha, \alpha \in I$ is also
$H$-closed.

Take $T_\alpha=(A_\alpha)^L_H$ and check that $\bigcap_\alpha
(A_\alpha)^L_H = (\bigcup_\alpha A_\alpha)^L_H$. Let $u \in
\bigcap_\alpha T_\alpha =T$. The inclusions $u \in T_\alpha$ and
$A_\alpha \subset Val^X_{H}(u)$ always hold true. So,
$\bigcup_\alpha A_\alpha \subset Val^X_{H}(u)$ and, hence $u \in
(\bigcup_\alpha A_\alpha)^L_H$.

Let $u \in (\bigcup_\alpha A_\alpha)^L_H$.
Then $\bigcup_\alpha A_\alpha \subset Val^X_{H}(u)$. The inclusions
$A_\alpha \subset Val^X_{H}(u)$ and $u \in (A_\alpha)^L_H =
T_\alpha$ always hold true. Hence, $u \in T$.
\end{proof}

From this theorem follows  that if the algebras $H_1$
and $H_2$ are $LG-$equivalent, then they are elementary equivalent (see \cite{PZ}).
Besides, if $H_1$ and $H_2$ are isotyped, then the categories of
elementary sets $LK_\Theta(H_1)$ and $LK_\Theta(H_2)$ (\cite{PZ}) are 
isomorphic (see also Section 5).

\begin{remark}
Our definition of a type corresponds  to the notion of a complete type in Model theory.  A complete type is an $H-$closed ultrafilter defined by a single point. 
However, an arbitrary   $H-$closed set is defined by a set of points. In this case also there are  relations with the model theoretic general theory of types. In particular, the $H-$closure of an arbitrary  type is always an intersection of complete types.
\end{remark}

\subsection{Infinitary logic}
Let us make some remarks on the relations with infinitary logic. We
start from quasiidentities. Take a binary relation $T$ in $W=W(X)$.
Consider a formula  (quasiidentity) $$(\bigwedge_{(w,w')\in T} (w \equiv w')) \to
(w_0 \equiv w'_0).\qquad\qquad (*)$$ If $T$ is infinite, it is an infinitary
quasiidentity.

Let, further, $H$ be an algebra in $\Theta$ and consider 
$T''_H$. It is proved in \cite{P5} that $(w_0, w'_0) \in T''_H$ if and
only if the quasiidentity  $(*)$ holds in $H$.

Proceed now from $T \subset \Phi(X)$ and consider a formula
$$(\bigwedge_{u \in T} u) \to v, \qquad\qquad (**)$$ where $v \in \Phi(X)$. The formula $(**)$ is
infinitary if T is an infinite set. We write for short $T \to v$. This
formula holds in $H$ if and only if $v \in T^{LL}_H$ \cite{P8}. Denote
by $Im(Th)(H)$ the set of all formulas of the form $T \to v$, holding in
$H$. It is the {\it implicative theory} of the algebra $H$. We may say that
the algebras $H_1$ and $H_2$ are $LG$-equivalent (isotyped), if and only if
their implicative theories coincide, i.e., $Im(Th)(H_1)=Im(Th)(H_2)$ (see \cite{PZ}).
The presence here of infinitary formulas is naturally stipulated by the aims of universal algebraic geometry. Such
formulas are not new also in Model Theory. In particular, they
participate in the theory of abstract elementary classes
($AEC$-classes) of models \cite{G}.

\subsection{Galois correspondence and morphisms in $Hal_\Theta^0$}

Let us show the relation between the Galois correspondence and morphisms in the category $Hal_\Theta^0$.
Let $s: W(X)\to W(Y)$ and $s_*: \Phi(X)\to \Phi(Y)$ be given. For every $u\in \Phi(X)$ we have
$$
Val^Y_H(s_*u)=s_*Val_H^X(u),
$$
Here, $\mu\in s_*Val_H^X(u)$ if $\mu s\in Val^X_H(u)$.

For $T\subset \Phi(X)$ denote by $s_*T$ the subset in $\Phi(Y)$ which consists of $s_*u$, $u\in T$. We have
  $$(s_*T)^L_H=s_*T^L_H,$$
   \noindent
  (see \cite{P7}). Thus, if $A=T^L_H$ is an elementary set in $Hom(W(X),H)$ then $s_*A$ is an elementary set in
$Hom(W(Y),H)$. As usual $\mu\in s_*A$ if $\mu s\in A$.

Connect now $X$ and $Y-$types  over $H$ for the given $s: W(X)\to W(Y)$ and $s_*: \Phi(X)\to \Phi(Y)$. For each point $\nu: W(X)\to H$ we  denote by $LKer^X(\nu)$ its logical kernel, having in mind that this kernel is calculated in $\Phi(X)$. Let us check that
$$s_*LKer^X(\mu s)\subset LKer^Y(\mu)$$
\noindent
  for   $\mu: W(Y)\to H$.

Let $u\in \Phi(X)$ and let $u\in LKer^X(\mu s)$. This gives $\mu s \in Val^X_H(u)$. Thus $\mu \in s_*Val^X_H(u)=Val^Y_H(s_*u)$, i.e., $s_*u\in LKer^Y(\mu)$. The inclusion is checked.

Apply once again the $L$-transition. We have:
 $$
 (LKer^Y(\mu))^L_H=B\subset (s_*LKer^X(\mu s))^L_H=s_*(LKer^X(\mu s))^L_H=s_*A.
$$
Here, $B$ is the closure of the point $\mu$, $A$ is the closure of the point $\mu s$ and
$B\subset Hom(W(Y),H)$, $A\subset Hom(W(X),H)$, $s_*A\subset Hom(W(Y),H)$, and $B\subset s_*A$.
$B$ and $A$ are minimal elementary sets (i.e., they do not contain other elementary sets), while $s_*A$ is not necessarily a minimal set.

\subsection{Relations $\rho$ and $\tau$}

Along with the relation $\rho$ consider relations $\rho_0$ and
$\tau$ on the given set $Hom(W(X),H)$. The relation $\rho_0$ is determined by the
decomposition of $Hom(W(X),H)$ into orbits of the group $Aut(H)$.
The inclusion $\rho_0 \subset \rho$ always holds, but we are
interested in the situation of the equality $\rho_0 = \rho$.

Let $X=\{x_1, \ldots , x_n \}$ and let $$(\mu(x_1), \ldots , \mu(x_n)) =
\bar a = (a_1, \ldots , a_n),$$  $$(\nu(x_1), \ldots , \nu(x_n)) =
\bar b = (b_1, \ldots , b_n)$$ for the points $\mu$ and $\nu$.
Denote by $A$ and $B$ subalgebras in $H$, generated by all $a_1,
\ldots , a_n$ and all $b_1, \ldots , b_n$ respectively. We set: $\mu
\tau \nu$ if and only if the transitions $a_i \to b_i$ determine the
isomorphism $\eta : A \to B$. We have $\rho _0 \subset \tau$.
Actually the following theorem takes place:
\begin{theorem}
The condition $\mu \tau \nu$ holds true if and only if $AtKer(\mu)=
AtKer(\nu)$ in $\Phi(X)$ or what is the same $Ker(\mu)=
Ker(\nu)$ in $W(X)$ .
\end{theorem}
\begin{proof}
Let $\mu \tau \nu$ hold and the corresponding isomorphism $\eta:A
\to B$ be given. Take $w \equiv w' \in AtKer(\mu)$. We have $w^\mu =
{w'}^\mu$ in $A$ and $w^{\mu\eta} = {w'}^{\mu\eta}$ in $B$. The
equality $w^\nu = {w'}^\nu$ holds in $B$ and, hence, $w \equiv w'
\in AtKer(\nu)$. Similarly, if $w \not \equiv w' \in AtKer(\mu)$,
then $w \not \equiv w' \in AtKer(\nu)$. Thus $AtKer(\mu) \subset
AtKer(\nu)$. The second direction is similar.

Let us check the converse statement. Consider homomorphisms $\alpha:W(X) \to A$ and
$\beta:W(X) \to B$, where $X=\{x_1, \ldots , x_n \}$,
$\alpha(x_i)=a_i=\mu(x_i)$, $\beta(x_i)=b_i=\nu(x_i)$. Let $w(x_1,
\ldots , x_n) =w'(x_1, \ldots , x_n )$ lie in the kernel $Ker(\alpha)$.
Then $w\equiv w' \in AtKer(\mu)= AtKer(\nu)$. Therefore, $w\equiv w'
\in Ker(\beta)$. More precisely, $w\equiv w' \in Ker(\alpha)$ if and
only if $w\equiv w' \in Ker(\beta)$, i.e., $ Ker(\alpha)=
Ker(\beta)$. We get an isomorphism $A \to B$ induced by
 $a_i \to b_i$.
\end{proof}

It is
clear that $LKer(\mu)= LKer(\nu)$ implies $AtKer(\mu)= AtKer(\nu)$
which also means that $\rho \subset \tau$. This gives us $\rho_0\subset\rho \subset \tau$ and if $\tau = \rho_0$ for
the given $X$ and $H$ then $\rho = \rho_0$.

Let us show that every coset of the relation $\tau$ is an elementary set.
Let $T$ be an arbitrary congruence on the algebra $W=W(X)$. Consider the set of formulas 
in $\Phi(X)$  defined by
 $T' =T'_1 \cup T'_2$, where $T'_1= \{w \equiv w' |
(w,w') \in T \}$ and  $T'_2= \{w \not\equiv w' |(w,w') \not\in T \}$.
It is easy to understand that a point $\mu : W(X) \to H$ satisfies
the set $T'$, that is  $\mu \in (T')^L_H$, if and only if $Ker(\mu)=T$.
Hence, the coset of the relation $\tau$ containing the point $\mu$ is an  elementary set defined by the set of formulas $T'$ with $Ker(\mu)=T$.
 We will
always use this remark  dealing with the relation $\tau$.

Note that the main future problem in this paper is to find the conditions on algebra $H$
which provide isomorphism $\eta :A \to B$ to be realized by some
automorphism of the algebra $H$. In this case we have $\tau=\rho_0$.

\subsection{Isotypeness of points}

Along with the notion of isotypeness of algebras we
introduce also the notion of  isotyped  points over algebras.
We say that {\it two
$X$-points $\mu, \nu: W(X) \to H$ are isotyped} if $LKer (\mu)= LKer
(\nu)$ which means that $\mu \rho \nu$ holds true. 

This notion lead to the definition
of  {\it logical isotypeness also for  elements} of the algebra $H$. We say that the elements $a_1$ and $a_2$  of an algebra $H$ are isotyped if the corresponding points $\mu, \nu\in Hom(W(x),H)$  with $\mu(x)=a_1$ and $\nu(x)=a_2$ are isotyped. Actually the notion of isotyped elements of $H$ is related to the following  idea.

 We proceed from a formula or a set
of formulas  $T$. 
Let the elements $a_1$ and $a_2$ from $H$ with the corresponding points $\mu, \nu\in Hom(W(x),H)$   be isotyped. Then the points
$\mu$ and $\nu$ are isotyped and let $T$ belong to $LKer(\mu)$. Then $T$
belongs to $LKer(\nu)$ as well. So, both points satisfy $T$ and let formulas from $T$ describe some algebraic property of elements in $H$. Isotypeness  of $a_1$ and $a_2$ means that  these elements both satisfy this algebraic property.
 For example,
 if $g$ and $g'$ are isotyped and $g$ has a finite
order $n$, then  $g'$ has the same order. Here $T$ consists of the formula $x^n
\equiv 1$.

Consider another example of engel elements and nil-elements in groups.
Recall that an element $g \in H$ is an $n$-engel element, if  there is $n=n(g)$ such that $[a,g,\ldots ,g]=1$ for any $a
\in H$. Here $[a,g,\ldots ,g]$ stands for a  composite commutator, where
$g$ is repeated $n$ times. Thus, the point $\mu$ with $\mu(x)=a$, $\mu(y)=g$ satisfies the formula
$\forall x([x,y, \ldots ,y]=1)$ if and only if $g$ is an $n$-engel
element. This formula is $T$
 .

Proceed further from $\mu, \nu : W(x,y) \to H$, $\mu(x)=a$, $\mu(y)
=g$, $\nu(x)=a'$, $\nu(y) =g'$. Let $\mu$ and  $\nu$ be isotyped (i.e., $g$ and $g'$ are isotyped) and
the element $g$ be an $n$-engel one. Then $g'$ is engel  as well. These remarks imply that if $H$ is
noetherian group and $g$ belongs to its nilpotent radical, then $g'$
belongs to it as well. See
\cite{Ba} and \cite{P2}.The similar fact is true for the solvable radical of a
noetherian group for the corresponding $T$ and isotyped $g$ and $g'$ (see \cite{P10}).

Let us pass to nil-elements. An element $g$ of a group $G$ is a
nil-element, if for every $a \in G$ there exists $n = n(a,g)$ such that  $[a, g, \ldots ,g] =1$,
where $g$ is taken $n$ times. The
property of being nil-element is not expressed as a formula, since
the definition uses a quantifier of the type $\exists n$ for natural
$n$.

We can improve the situation in the following way. Fix two elements $g$ and $g'$ in
the group $G$ and consider all possible pairs $(a,g)$ and $(a', g')$ where $a$ and $a'$ are one-to-one
related. Suppose that  the corresponding $\mu, \nu :
W(x,y) \to H$ are isotyped. Assume now that $g$ is a nil-element. Then the pair $(a,g)$ satisfies
an equation $[x, y, \ldots , y]=1$. Isotypeness of $\mu, \nu$ implies that the  pair $(a',g')$ satisfies the
same equation. This holds for every appropriate $a$ and $a'$, and,
hence, $g'$ is also nil-element. 

Let now $G$ be a solvable group (or, more generally, radical
group (\cite{P1})), and  $g$ and $g'$ be its isotyped elements. Then, if one of them
belongs to a locally nilpotent radical, then the second one does. It
follows from the previous remarks and from \cite{P1}. We can also
consider various other radicals in other groups. A lot of natural problems arise in this way.

\section{Logical noetherianity} \label{sec:LN}
\subsection{Definitions and problems}\label{sub:dfn}
Let $H$ be an algebra from $\Theta$. Consider three conditions of noetherianity for $H$.

1. An algebra $H$ is {\it logically noetherian}, if for any finite $X$ an arbitrary
elementary set $A \subset Hom(W(X),H)$ is {\it finitely definable}. This means that if $A=T^L_H$ then there exists a finite set $T_0$ (not necessarily a subset of $T$) such that $A=(T^0)^L_H$.

2. An algebra $H$ is {\it strictly logically noetherian}, if for any finite $X$ and arbitrary
elementary set $A =T^L_H\subset Hom(W(X),H)$ there exists a finite set $T_0\subset T$ such that $A=(T_0)^L_H$.

3. An algebra $H$ is {\it weakly logically noetherian} if for any $T \to v$ which holds in $H$ there exists a finite subset $T_0\subset T$ (possibly depending on $v$) such that $T_0 \to v$ holds in $H$.


For each of these noetherian conditions isotypeness of the algebras $H_1$
and $H_2$  is equivalent to their
elementary equivalence.

Clearly, every finite algebra is logically
noetherian in any sense. In Model Theory there exist examples of infinite logically noetherian algebras  (see also Proposition \ref{prop:inf_ab_gr}).

We will use now the following remark. All elementary
sets in the given $Hom(W(X),H)$ constitute a
sub-lattice in the lattice of all subsets. Dually, we have a lattice
of all $H$-closed filters in $\Phi(X)$. 
Strict
noethrianity of $H$ is equivalent to the noetherianity condition
of the corresponding lattice of filters and, what is the same, to the
condition of artinianity of the lattice of elementary sets for a finite
$X$.

The principal problem related to  Model Theory  is as follows:

\begin{problem}
To develop a general approach for constructing  logically noetherian algebras $H$ in different varieties $\Theta$.
\end{problem}

Here $\Theta$ can be the variety of groups, the variety of abelian groups, the variety of commutative associative algebras. For example, we have a problem to describe the logically noetherian abelian groups.

Let us make some remarks on algebras $H$ with the finite set of types  $S^X(H)$ 
for every $X$. It follows from the definitions that if the
set $S^X(H)$ is finite, then we have also a finite number of
different $H$-closed filters in $\Phi(X)$. We have also a finite
number of elementary sets in $Hom(W(X),H)$. This gives artinianity
and noetherianity of the corresponding lattices and strict
noetherianity of the algebra $H$. It is clear that if we have a
finite number of $Aut(H)$-orbits in $Hom(W(X),H)$, then the set
$S^X(H)$ is finite as well.

\subsection{Automorphic finitarity of algebras}

In this subsection we consider the question raised in Problem \ref{sub:dfn}.

\begin{defn}
Algebra $H$ is called automorphically ($Aut(H)$)-finitary, if there is a finite number of
$Aut(H)$-orbits in $Hom(W(X),H)$ for
any finite $X$.
\end{defn}

If $H$ is $Aut(H)$-finitary, then there is a finite number of types on $H$ and hence a finite number of elementary sets. Thus, $H$ is strictly logically noetherian.  So, there arise a question whether an algebra $H$ is automorphically finitary.

We shall specify the problem above for the case of groups:

\begin{problem}
To develop general methods for constructing  infinite automorphically
finitary groups.
\end{problem}

Consider an example.

\begin{proposition}\label{prop:inf_ab_gr}
Let $H$ be an infinite abelian group with the identity $x^p\equiv 1$
with prime $p$. Such a group is $Aut(H)$-finitary.
\end{proposition}

\begin{remark} This proposition is not unexpected in view of  the well known model theoretic result by  Ryll-Nardzewski  (see \cite{CK},\cite{Ma}) which basically states that a complete theory is  $\aleph_0$-categorical if and only it has a finite number of types.
However, for the sake of
completeness we present an independent simple proof of the
proposition \ref{prop:inf_ab_gr}.

\end{remark}

\begin{proof}
Proceed from the variety $\Theta$ of all abelian groups with the
identity $x^p\equiv 1$,  $p$ is prime. Fix $X=\{x_1, \ldots , x_n\}$
and let $W=W(X)$ be a free in $\Theta$ group over $X$. This group is
finite. Let $T$ be a subgroup of $W$.

Take a formula $u = (\bigwedge_{u_i \in T} (u_i \equiv 0)) \wedge
(\bigwedge_{v_i \not \in T} (v_i \not \equiv 0))$. Denote
$Val^X_H(u) =V$. Obviously, a point $\mu :W\to H$ belongs to $V\subset Hom(W(X),H) $ if
and only if $Ker(\mu) =T$.

Let now $\mu$ and $\nu$ be two points in $V$, $A$ and $B$ their
images in $H$. Then $A=<a_1, \ldots , a_n>$, $B=<b_1, \ldots ,
b_n>$. Since $Ker(\mu) =Ker(\nu)$, we have a commutative diagram


$$
\CD
A @[2]> \sigma  >> B\\
 @[2]/NW/\mu//@.@.\;    @/NE//\nu/\\
 @.W(X) \\
\endCD
$$


Here $\sigma$ is an isomorphism, $\sigma(a_i)=b_i$ and $\mu \sigma =
\nu$. Subgroups $A$ and $B$ have complements in $H$, i.e.,
$A\bigoplus A' = B\bigoplus B' =H$. Since $H$ is infinite, $A'$ and
$B'$ are infinite and isomorphic. This means that the isomorphism
$\sigma: A \to B$ can be extended up to an automorphism $\sigma \in
Aut(H)$. Therefore $\mu$ and $\nu$ are conjugated by this
automorphism. Besides, sequences $(a_1, \ldots , a_n)$, $(b_1,
\ldots , b_n)$ are also conjugated by the automorphism $\sigma$.
This means that the set $V$ belongs to $Aut(H)$-orbit, determined by
each of the points $\mu$ and $\nu$. On the other hand, since $V$ is
an elementary set, the pointed orbit is contained in $V$. Therefore,
$V$ is equal to the  orbit containing $\mu$.

Every orbit, determined by a point $\mu : W \to H$ is of this form.
Indeed, take $T = Ker(\mu)$ and a formula $u$ constructed by $T$.
Then $V=Val^X_H(u)$ is an $Aut(H)$-orbit, determined by the point
$\mu$.

Since there are finite sets of different $T$ and $u$, we have a
finite set of orbits. Moreover, all of them are one-defined elementary sets.
The group $H$ is an infinite strictly logically  noetherian group.

\end{proof}

Note that we use in this proof  a representation of the
point $\mu\in H^{(n)}$ as a point - homomorphism $\mu : W \to H$.

A similar situation holds in finitely dimension vector spaces over
arbitrary finite field. If the field is infinite, then the
description of orbits is the same, but the number of orbits is  infinite.
 Various other examples of infinite automorphically
finitary (and thus strictly logically noetherian) groups  were suggested to me by
A.Olshansky (algebraic approach) and B.Zilber (model theoretic
approach). In their constructions the groups are far from being abelian.

In this concern there is a
problem whether there exist infinite noetherian groups which are
also logically noetherian. Noetherian polycyclic groups and their
finite extensions are not the case. Infinite cyclic group is also not logically noetherian.

Let us conclude this section with one useful remark.


Suppose that $H_1$ and $H_2$ are isotyped and $H_1$ is strictly logically
noetherian.  Then $H_2$ is strictly logically noetherian too.

Indeed, let $H_1$ and $H_2$ be isotyped and $H_1$ be strictly logically noetherian. Then for $T\subset \Phi(X)$
and some finite part $T_0 \subset T$ we have $T^{LL}_{H_2}=T^{LL}_{H_1}={T_0}^{LL}_{H_1}=
{T_0}^{LL}_{H_2}$ and, hence, $T^{LL}_{H_2}={T_0}^{LL}_{H_2}$. Therefore, $H_2$ is strictly
logically noetherian.




\section{Isotypeness and isomorphism}
\subsection{Separable algebras}
Another general problem is to study relations between isotypeness property and isomorphism of algebras.
First of all, there are various examples of non-isomorphic isotyped algebras. Even for the case of fields
there are  examples of such kind .  
 We will present examples of isotyped but not isomorphic algebras in Subsection \ref{ex}.


\begin{defn}\label{sep}
An algebra $H\in \Theta$ is called separable in $\Theta$ if  each
$H' \in \Theta$ isotyped to $H$ is isomorphic to $H$.
\end{defn}


\begin{remark}
Sometimes it is worth to modify Definition \ref{sep} and to consider only finitely-generated $H'$.
\end{remark}

\begin{problem}
For which $\Theta$ every free in $\Theta$ algebra $W=W(X)$ with
finite $X$ is separable?
\end{problem}

In other words this problem asks when every free in $\Theta$ algebra $W=W(X)$ with
finite $X$ can be distinguished in $\Theta$ by  means of logic of types (i.e.,  LG-logic).

The problem is stated for an arbitrary variety $\Theta$, but most of
all we are interested in the  variety of groups $\Theta=Grp$ and the
variety $\Theta=Com-P$ of commutative and associative algebras over
a field $P$.

Z.Sela 
(unpublished) showed that free noncommutative groups $F(X)$ and $F(Y)$ are isotyped
if and only if they are isomorphic.

\subsection{Examples of non-isomorphic isotyped algebras}\label{ex}

\begin{proposition}\label{prop:inf_dim}
Let algebras $H_1$ and $H_2$ be infinitely dimension vector spaces over a
field $P$. Then they are isotyped.
\end{proposition}

\begin{proof}
We use a method which can be applied in other cases as well.
Proceed from a variety $\Theta$ of vector spaces over $P$.
Take $X=\{x_1, \ldots , x_n \}$ and let $W=W(X)$ be the corresponding free object, i.e., an arbitrary linear space of dimension $n$.
Consider points $\mu:W \to H_1$ and  $\nu:W(X)\to H_2$. A sequence $\bar a=( a_1, \ldots, a_n)$,
$a_i= \mu (x_i)$ corresponds to the point $\mu$. Similarly, we have  $\bar b=( b_1, \ldots, b_n)$
for $\nu$. Denote $A=< a_1, \ldots, a_n >$ and $B= <b_1, \ldots, b_n >$.
The points $\mu$ and $\nu$ we call isomorphic, if there is an isomorphism $\alpha : A \to B$
with $\alpha (a_i)=b_i$. We write $\alpha \mu = \nu$. We have also $\mu = \alpha^{-1} \nu$.
If $H_1=H_2$ then isomorphism of the points $\mu$ and $\nu$   means  that $\mu$ and $\nu$
satisfy the relation $\tau$.

Along with the free algebra $W=W(X)$ consider the algebra of formulas
$\Phi =\Phi(X)$. A formula $u \in \Phi$ is called correct, if for
 isomorphic $\mu$ and $\nu$ the inclusion $\mu \in Val_{H_1}(u)$, i.e.,
$u \in LKer(\mu)$, holds if and only if $u \in LKer(\nu)$, i.e., $\nu \in
Val_{H_2}(u)$.

We intend to check that in our case every formula $u$ is correct. It is easy to see that for an arbitrary  $\Theta$ all
the equalities
$w \equiv w'$ are correct,  if $u$ is correct, then its negation $\neg u$
is correct, and  if $u_1$ and $u_2$ are correct, then $u_1 \vee u_2$ and
$u_1 \wedge u_2$ are correct as well. 


 However the implication if $u$ is correct, then so is
 $\exists x_i u$ for every $x_i$ is not probably true for arbitrary $\Theta$. Here there arises a question to find an example when the correctness condition is not fulfilled. We shall check that in the situation under consideration the implication is valid.

Without loss of generality take $x_i=x_1$. So, let $\mu$ and $\nu$
be isomorphic, $\nu = \alpha \mu$, $\mu \in Val_{H_1}(\exists x_1 u)
= \exists x_1 Val_{H_1}(u)$. There exists $\mu_1 \in Val_{H_1}(u)$
with $\mu (y) = \mu _1 (y)$ for each $y \not = x_1$, $y \in X$.

Take $a=\mu_1(x_1)$. For $\mu_1$ we have a sequence $(a, a_2,
\ldots, a_n )$.  Recall that $\bar a= (a_1, a_2, \ldots, a_n )$ and
$\bar b= (b_1, b_2, \ldots, b_n )$, $b_i=\alpha '(a_i)$, for $\mu$
and $\nu$ respectively. Take $A_1= <a, a_2, \ldots, a_n >= <a,< a_2,
\ldots, a_n >>$. We want to investigate an isomorphism $\alpha_1:
A_1 \to B_1$ with $B_1= <b, < b_2, \ldots, b_n >>$, where $b_1
=\alpha _1(a)$, $b_i =\alpha_1(a_i) = \alpha(a_i)$, $i=2, \ldots,
n$. The isomorphism $\alpha: A \to B$ induces an isomorphism
$\alpha': < a_2, \ldots, a_n > \to < b_2, \ldots, b_n >$. Suppose
first that $ a \in < a_2, \ldots, a_n >$. In this case we have an
isomorphism $\alpha_1 : A_1 \to B_1$ with $\alpha_1(a)=
\alpha'(a)=b$. Let $a\notin < a_2, \ldots, a_n >$. Then
$A_1$ is a vector space with the  dimension greater by one than the
dimension of the space $< a_2, \ldots, a_n >$. Take an arbitrary $b
\in H_2$  which does not lie in $< b_2, \ldots, b_n >$. We have a
vector space $B_1=<b, < b_2, \ldots, b_n >>$ of the same dimension
as  $A_1$. Assuming $\alpha _1(a)=b$ we determine an isomorphism
$\alpha_1 : A_1 \to B_1$, extending the isomorphism $\alpha'$.

Take further $\nu_1:W \to H_2$ defined by the rule $\nu_1(x_i)= \alpha_1
\mu_1(x_i)$, $i=1, \ldots, n$. Here $\mu_1$ and $\nu_1$ are
isomorphic. Since $\mu_1 \in Val_{H_1}(u)$, then $\nu_1 \in
Val_{H_1}(u)$ due to correctness of
 the formula $u$. The points
$\nu$ and $\nu_1$ coincide on the variables $x_2, \ldots ,x_n$ by
the construction. Hence, $\nu \in  Val_{H_2}(\exists x_1 u)= \exists
x_1 Val_{H_2}(u)$. Similarly, if $\nu \in \exists x_1 Val_{H_2}(u)$,
then $\mu = \alpha^{-1}\nu \in Val_{H_1}(\exists x_1 u)$.

Show that if
$u \in \Phi(Y)$ is correct then the formula $s_* u\in
\Phi(X)$ where $s: W(Y) \to W(X)$ and $s_* : \Phi(Y) \to \Phi(X)$ (see Section \ref{sec:appendix} for the
definition of the mapping $s_*$) is correct as well. Let $\mu$ and $\nu$ be isomorphic by
the isomorphism $\alpha$, and $\mu \in Val_{H_1}(s_* u) = s_*
Val_{H_1}(u)$. Here $\mu s \in Val_{H_1}(u)$. Apply the isomorphism
$\alpha$ with $\nu= \alpha \mu$. This gives the isomorphism $(\alpha
\mu)s = \nu s=\alpha (\mu s) $, and $\nu s$ and $\mu s$ are
isomorphic.

Suppose $u$ is correct. Then $ \nu s \in Val_{H_2}(u)$ and $\nu \in
s_* Val_{H_2}( u)= Val_{H_2}(s_* u)$. The formula $s_* u$ is also
correct. Using the definition of the algebra $\Phi(X)$ 
and the fact that all the equalities are correct,
we may conclude that all $u \in \Phi(X)$ are correct. If $\mu$ and
$\nu$ are isomorphic, then for any $u \in \Phi(X)$ we have $\mu \in
Val_{H_1}( u)$ if and only if $\nu \in Val_{H_2}( u)$.

Let now $\mu$ and $\nu$ be isomorphic. Since $u$ is correct we have $u \in LKer(\mu)$ if and
only if $u \in LKer(\nu)$. Thus $LKer(\mu)=LKer(\nu)$. If the dimensions of $H_1$ and $H_2$ are infinite, for any point
$\mu$ we can construct a point $\nu$ isomorphic to it. 
Hence, for each point $\mu\in Hom(W(X),H_1)$  there is $\nu\in Hom(W(X),H_2)$
with $LKer(\mu)= LKer(\nu)$.

The opposite is also true. This means that $H_1$ and $H_2$ are
isotyped. It is clear that they are not necessarily isomorphic.
\end{proof}

The  method from the proposition above can be used also in the case when $H_1$ and $H_2$
are infinite abelian groups of the finite exponent $p$. From the
other hand, what can be said if $H_1$ and $H_2$ are free abelian
groups of infinite range, or free noncommutative groups of infinite
range? Using considerations similar to those from Proposition \ref{prop:inf_dim}
we may, in particular, study locally cyclic torsion-free groups
to see if they are isotyped. This also should give examples of isotyped
but not isomorphic algebras.

\begin{remark}
In fact, the proof of Proposition \ref{prop:inf_dim} follows already from the freeness of vector spaces. We gave a detailed proof having in mind applications to other situations.
\end{remark}

\section{Logically perfect algebras}
\subsection{Embedding of algebras}

Remind that we defined equivalence $\rho$ on the set (affine space)
$Hom(W(X),H)$. If $\mu$ and $\nu$ are two points $W(X) \to H$, then
$\mu \rho \nu$ if and only if $LKer(\mu)= LKer(\nu)$(see \ref{sub:lg}). These logical
kernels are calculated in the algebra of formulas $\Phi(X)$. A group
$Aut(H)$ acts on the set $Hom(W(X),H)$ by the rule $\mu \to \mu
\sigma$, $\sigma \in Aut(H)$. Here each elementary set in
$Hom(W(X),H)$ is invariant under the action of the group $Aut(H)$.
This implies that $\mu \rho (\mu \sigma)$ always hold true.

\begin{defn}
An algebra $H \in \Theta$ is called {\it logically perfect} if for any $X$
the relation $\mu \rho \nu$ holds if and only if the points $\mu$ and $\nu$ are
conjugated by an automorphism of the given $H$.
\end{defn}
This condition means that for any $X$ every coset of the relation
$\rho$ is an orbit of the group $Aut(H)$.  Hence, for logically perfect groups every orbit is an elementary set.

Let us formulate the following problem:
\begin{problem}
Consider conditions when a given algebra $H \in \Theta$ can be
embedded into a perfect algebra $H' \in \Theta$.
\end{problem}

This problem is closely related to the following well-known  results of
Model Theory \cite{Ma}:

1. For every finite set $X$, algebra $H \in \Theta$, and points $\mu, \nu : W(X)
\to H$ the condition $LKer(\mu)= LKer(\nu)$ is equivalent to the following one:
for an elementary embedding $H \to G \in \Theta$ there is $\sigma \in
Aut(G)$ with $\mu \sigma = \nu$.

2. There exists a large algebra $H \in \Theta$, such that for
every $X$ and points $\mu, \nu : W(X) \to H$ we have $LKer(\mu)=
LKer(\nu)$ if and only if $\mu \sigma = \nu$ for some $\sigma \in
Aut(H)$.

Both these theorems are theorems of existence and  rely on the compactness theorem. So they are highly non-constructive.
Our interest is a bit different. 
We look at the specific varieties $\Theta$. The question is how to realize the transition from $H$
to $H'$ by
constructions
in $\Theta$ (see, for example, Proposition 4.2, where $\Theta$ is the variety of groups).

 So, we want to find out  how logically perfect algebras look like. The first question is 


\begin{problem}\label{orb}
For which algebras $H \in \Theta$ every $Aut(H)$-orbit  in
$Hom(W(X),H)$  is an elementary set for any finite $X$?
\end{problem}

It was noted earlier that if $H$ is logically perfect then $Aut(H)$ orbits are the elementary sets. In fact, the opposite statement is true as well (Proposition \ref{perf}). This fact explains the importance of  Problem \ref{orb}.
\begin{prop}\label{perf}
An algebra $H$ is logically perfect if and only if every $Aut(H)$-orbit in $Hom(W(X),H)$ is an elementary set for every $X$.
\end{prop}

\begin{proof} Let us fix a finite $X$ and take a point $\mu: W(X)\to H$. Let $A$ be an $Aut(H)$-orbit defined by $\mu$, $\mu\in A$. Suppose that $A$ is an elementary set. Then $A^{LL}_H=A$.
We have $A^L_H\subset \{\mu\}^L_H=LKer(\mu)$. Then $\{\mu\}^{LL}_H= LKer(\mu)^L_H\subset A^{LL}_H=A$. The point $\mu$ belongs to the elementary set $\{\mu\}^{LL}_H$ and thus the whole orbit $A$ lies in $\{\mu\}^{LL}_H$. We get the equality $\{\mu\}^{LL}_H=A$. We have also
$$
\{\mu\}^{LLL}_H=A^L_H=\{\mu\}^L_H=LKer(\mu).
$$
Thus, for every orbit $A$ we have  $A^L_H=LKer(\mu)$, where $\mu$ is a point from $A$.  We used that all orbits $A$ are elementary sets. Different types correspond to different orbits over $H$ and such correspondence exhausts all types related to $H$.

Let now $\mu\rho \nu$ and let $A$ be an $Aut(H)$-orbit over $\mu$, $B$ an $Aut(H)$-orbit over $\nu$. We have $LKer(\mu)=LKer(\nu)$ i.e., $A^L_H=B^L_H$, and thus $A=B$ since $A$ and $B$ are elementary sets.  This means that the points $\mu$ and $\nu$ belong to a common orbit. This gives $\mu\rho_0\nu$. Hence $\rho=\rho_0$ and thus the algebra $H$ is perfect.
\end{proof}

Consider the notion of strictly logically perfect algebras.

\begin{defn}
An algebra $H \in \Theta$ is called {\it strictly logically perfect} if for any  $X$
and points $\mu: W(X)\to H$ and $\nu: W(X) \to H$  the condition $LKer^0(\mu)=LKer^0(\nu)$ implies  $\mu = \nu \sigma$ for some $\sigma \in Aut (H)$.
\end{defn}
This notion is indeed more strict than being logically perfect in the usual sense. As it was done before one can prove that an algebra $H$ is strictly logically perfect if for any $X$ every $Aut(H)-$orbit in $Hom(W(X),H)$ is an elementary set defined by some set of formulas $T\subset \Phi^0(X).$ In this case we say that every orbit is a strictly elementary set.



\subsection{Method of $HNN$-extensions}
This subsection describes the case $\Theta = Grp$. We use here the
method of $HNN$-extensions (see \cite{HNN}).

Let $a_i$, $b_i$, $i \in I$ be two sets of elements in the group
$H$, $A$ and $B$ be subgroups in $H$, generated by all $a_i$ and
$b_i$ respectively. Consider a system of equations $t a_i t ^{-1}
=b_i$, $i \in I$. It is proved (Theorem 1 from \cite{HNN}) that such a
system has a solution for some $t$ belonging to a group $G$ containing
$H$, if and only if the subgroups $A$ and $B$ are isomorphic under
transition $a_i \to b_i$.

This theorem has a lot of applications. In particular, it implies
that for every torsion free group $H$ there exists a torsion free group $G$
containing $H$, such that there is only one non-trivial class of
conjugated elements in $G$. This means also the following. Let $G$
be such a group  and $W=W(x)$ a free
cyclic group. Consider the affine space $Hom(W,G)$ and an elementary
set $A$, defined by a single "equation" $x\not = 1$. 
As usual,
it is invariant under the action of the group $Aut(G)$. On the other
hand, any two elements $\mu$ and $\nu$ in the given elementary set
are conjugated by an element in $Aut(G)$. Hence, $A$ is the unique
nontrivial orbit of the group $Aut(G)$. We would like to build a
group $G$ in which something similar holds in
the affine space $Hom(W(X),G)$, for each finite $X$. Namely, we would like to construct a group $G$ by given group $H$ such that for any $X$ in the affine space $Hom(W(X),G)$ would be a finite number $Aut(G)-$orbits and each of the orbit should be an elementary set.

 Let us start from the situation when $H$ is a set without
operations.

Let us use the scheme from the proof of Proposition \ref{prop:inf_ab_gr}. Given an
equivalence $\tau$ on the set $X = \{x_1, \ldots, x_n\}$, consider
formulas $x_i \equiv x_j$ if $x_i \tau x_j$, and $x_i \not \equiv x_j$
otherwise. 
Let $u=u_\tau$ be a conjunction of all
such equalities and inequalities. Pass to $V=Val^X_H(u)$. A point
$\mu :X\to H$ belongs to $V$ if and only if $Ker(\mu)=\tau$. 
If $A$ is an image of the point $\mu$, then we have a
bijection $X / \tau \to A$. If $\mu$ and $\nu$ are two points in
$V$, then $Ker(\mu)=Ker(\nu)=\tau$, which gives a commutative
diagram


$$
\CD
A @[2]> \sigma  >> B\\
 @[2]/NW/\mu//@.@.\;    @/NE//\nu/\\
 @. X \\
\endCD
$$


Here $\sigma$ is a bijection and $B$ is the image of the point
$\nu$. The sets $A$ and $B$ are finite, while the set $H$ we regard
as infinite. We have $A\cup A' =H=B\cup B'$ where $A'$ and $B'$ are
of the same cardinality. This leads to the extension of the
bijection $\sigma$ up to a permutation $\sigma \in S_H$. According to
diagram $\nu = \mu \sigma$ and $\sigma(a_i)=b_i$, where
$a_i=\mu(x_i)$ and $b_i=\nu(x_i)$. Hence, every $V_\tau\tau
=Val^X_H(u_\tau)$ is an orbit of the permutation group $S_H$, and
these are all the orbits. Thus we obtained a finite number of  finitely defined orbits.


%
%

\subsection{Embedding theorem} In the proof of the embedding theorem we will use Theorem 2 from
\cite{HNN}:

Let $\eta_i : A_i \to B_i$, $i \in I$, be isomorphisms of subgroups
of the group $H$. There exists a group $G=<H,F>$, where $F$ is
freely generated by elements $t_i$ so that $\eta_i(a)=t_i a
{t_i}^{-1}$ for each $a \in A_i$.

Our next goal is the following:
\begin{theorem}\label{emb}
Each group can be embedded into a logically perfect group.
\end{theorem}
This theorem seems to be a particular case of a similar model theoretic
result. We present here an independent group theoretic proof, using
$HNN$-theory.
\begin{proof}
Let $\bar a = (a_1, \ldots , a_n)$ and $\bar b = (b_1, \ldots ,
b_n)$ be two  sequences of elements in $H$, and let  $A$ and $B$ be subgroups, generated by the elements
$a_1, \ldots , a_n$ and $b_1, \ldots , b_n$, respectively. According to Theorem 1 from
\cite{HNN}, there exists a group $G= G_{(\bar a, \bar b)}$ containing
$H$ and such that some automorphism of the group $G$ takes $\bar a$
into $\bar b$ if and only if this determines an isomorphism $\eta :
A \to B$. Having
this in mind, consider an equivalence $\tau$ 
defined in $H^{(n)}$, $n$ is fixed. Define $\bar a \tau \bar b$ if there is an
isomorphism $\eta = \eta _{(\bar a, \bar b)}: A \to B$, extending
correspondence  $a_i \to b_i$. If $\bar a \tau \bar b$ holds true, then
we have a group $G= G_{(\bar a, \bar b)}$ with the element
$t=t_{(\bar a, \bar b)}$, such that $t$ determines an inner
automorphism of the group $G$, taking $\bar a$ to $ \bar b$. The
group $G$ can be represented as $G=<H,t>$. The equivalence $\tau$ is
automatically extended to the space $Hom(W(X),H)$, $X=\{x_1, \ldots
, x_n\}$.

Consider further the relations $\tau_n$ for every natural $n$. The
relations $\tau_n$ are defined on all sequences  $\bar a = (a_1,
\ldots , a_n)$ of the length $n$. A subgroup $A$ in $H$, generated
by the elements $a_1, \ldots , a_n$, corresponds to every sequence.
We write $A=A(\bar a)$. As before, $\bar a \tau_n \bar b$ if the
transitions $a_i \to b_i$ determine the isomorphism $\eta_{(\bar a,
\bar b)} : A(\bar a) \to B(\bar b)$. Consider elements $t_{(\bar a,
\bar b)}$, which freely generate the group $F_n$. According to \cite{HNN}, we
have $G_n=<H,F_n>$ so that every $t_{(\bar a, \bar b)}$ determines
an inner automorphism of the group $G_n$ which, like $\eta_{(\bar a,
\bar b)}$, transforms $a_i$ into $b_i$.

On the next step we vary $n$ and consider the relations $\tau_n$ for different
$n$. Take different $t_{(\bar a, \bar b)}$ with $\bar a\tau_n \bar
b$ for all $n$. Generate by these $t_{(\bar a, \bar b)}$ the free
group $F$. Then we have a group $G=<H,F>$. If $\bar a\tau_n \bar b$,
then $t_{(\bar a, \bar b)}$ induces an isomorphism $\eta_{(\bar a,
\bar b)}: A(\bar a) \to B(\bar b)$. All this is valid due to
Theorem 2 from \cite{HNN}.

Denote the group $G=<H,F>$ by $H'$. Iterating the transition $H \to
H'$ we get the increasing sequence of groups $H$, $H'$, $H''$,\ldots, etc. Denote by  $H^0$ the union of all these groups. Consider sequences $\bar a = (a_1,
\ldots , a_n)$, where $a_i \in H^0$. For every $n$ if $\bar a
(\tau_n) \bar b$, then there exists an element $t \in H^0$, such
that the inner automorphism $\hat t$ transforms $\bar a$ into $\bar
b$. 

Show now that for every $X=\{x_1, \ldots , x_n\}$ in $Hom(W(X),H^0)$
the equality $\tau =\tau_n=\rho_0$ holds true and every coset of
the relation $\tau$ is $Aut(H^0)$-orbit.

Consider $Hom(W(X),H^0)$. Take two points $\mu, \nu :W(X) \to H^0$.
The condition $\mu \tau \nu$ means that $\mu \tau_n \nu$ for some
$n$, and if $(\mu(x_1), \ldots , \mu(x_n))= \bar a =(a_1, \ldots ,
a_n)$ and $(\nu(x_1), \ldots , \nu(x_n))= \bar b =(b_1, \ldots ,
b_n)$, then $\bar a \tau_n \bar b$. We have an inner automorphism
$\sigma=\hat t$, transforming $\bar a$ into $\bar b$ and,
simultaneously, $\mu$ into $\nu$. So, $\mu \tau \nu$ implies $\mu
\sigma = \nu$ for some $\sigma \in Aut(H^0)$, $\tau \subset \rho_0$.
Besides, $\rho_0 \subset \tau$ and $\tau =\rho_0$. We have also
$\tau = \rho$. Finitely, $\rho =\rho_0$ and the group $H^0$ is
perfect.
\end{proof}





\begin{remark}
All $Aut H^0-$orbits in $Hom(W(X),H^0)$ are elementary sets for every finite $X$. However the number of orbits is infinite and equals  to the number of different $Ker (\alpha)$ for the points $\alpha: W(X)\to H^0$. This group is not good enough for generalization of the example of a group with a single non-trivial conjugacy class \cite{HNN}, \cite{Ku}.

On the other hand, in Proposition \ref{prop:inf_ab_gr} we constructed a group $H$ with finite number $Aut H-$orbits and each of them is an elementary set. The group $H$ is logically perfect group.

\end{remark}

Note that a group $G$ is called {\it homogeneous} if every isomorphism $\eta: A\to B$ of  its finitely generated subgroups is realized by an inner automorphism of $G$. Such groups were constructed by Ph. Hall, O.Kegel, B.Neumann ( see \cite{Ha}, \cite{Ke}, \cite{Ne} ). It is easy to see that every group of such type  is logically perfect since here we have $\tau=\rho_0$ and $\rho=\rho_0$.
 Homogeneity property is considered in model theory with respect to an arbitrary algebraic system. In this case the automorphism of a system is not assumed to be inner. The notion of  logically perfect  algebra is close to the notion of homogeneous algebra. It is easy to see that the algebra $H$ is homogeneous if and only if for every $X$ we have $\tau=\rho_0$.
 We keep the term "perfectness" having in mind the observation that perfectness  in groups is provided, usually, by inner automorphisms (like in Theorem \ref{emb}).

In general,  a logically perfect algebra is  not necessarily homogeneous: it can happen that $\rho=\rho_0$ but $\tau\neq\rho$.  However, the following theorem takes place:

\begin{theorem} (G.I.Zhitomiskii)\label{zh}
An algebra $H$ in $\Theta$ is homogeneous if and only if $H$ is strictly logically perfect.
\end{theorem}

\begin{proof} We use the notion of a correct formula (see Proposition \ref{prop:inf_dim}). Recall this notion with respect to algebra $\Phi^0$. A formula $u \in \Phi^0$ is called correct, if for
any isomorphic $\mu$ and $\nu$ the inclusion $\mu \in Val_{H_1}(u)$, i.e.,
$u \in LKer^0(\mu)$, holds if and only if $u \in LKer^0(\nu)$, i.e., $\nu \in
Val_{H_2}(u)$.

If every formula $u\in \Phi^0(X)$ is correct then $LKer^0(\mu)=LKer^0(\nu)$, for the isomorphic $\mu,$ $\nu$.

Let now $H$ be strictly logically perfect. Using induction by cardinality of the set $X$ we will prove that every formula $u\in\Phi^0(X)$ is correct. Let $X=\{x\}$. Then every boolean formula over equalities is correct. The formula $\exists x u$ has 0 and 1 as values. So for one-element $X$ correctness takes place. Let $X=\{x_1,\ldots, x_n\}$ and let for any $X$ such that $|X|<n$ the property is proved. First of all note that the correctness property is preserved under application of the boolean operations. Let $u\in\Phi^0(X)$ be correct. Show that $\exists x_n u$ is correct. Let  the points $\mu$, $\nu:W(X)\to H$ be given, let $\mu_0$, $\nu_0$ be their restrictions on $X_0=\{x_1,\ldots, x_{n-1}\}$. Let $A$, $B$ be the images of the points $\mu$, $\nu$, respectively,  in $H$ and  $A_0$, $B_0$ be the corresponding images for $\mu_0$, $\nu_0$. Let $\alpha$ be an isomorphism $A\to B$. Denote by $\alpha_0$ the isomorphism $A_0\to B_0$ induced by $\alpha$. Take $\alpha \mu=\nu$ and $\alpha_0 \mu_0=\nu_0$. Since every formula in $\Phi^0(X_0)$ is correct we have $LKer^0(\mu_0)=LKer^0(\nu_0)$ for isomorphic $\mu_0$, $\nu_0$. Now, since $H$ is strictly logically perfect we have an automorphism $\sigma\in Aut H$ such that $\mu_0\sigma=\nu_0$. Take $\mu \sigma$ and $\nu$. Their restrictions to $X_0$ coincide.

Let now $\mu\in Val^X_H(\exists x_nu)=\exists x_n Val^X_H(u)$. Then there is $\mu'\in Val^X_H(u)$, such that $\mu$ and $\mu'$ coincide on $X_0$. Then $\mu\sigma$  and $\mu'\sigma$ coincide on $X_0$. But $\mu\sigma$ and $\nu$ coincide on $X_0$. Then
$\mu'\sigma=\nu'$ and $\nu$ coincide on $X_0$. Besides, along with $\mu'$ the point $\mu'\sigma=\nu'$ belongs to $Val^X_H(u)$.
Then $\nu\in Val^X_H(\exists x_nu)$. Hence $\mu\in Val^X_H(\exists x_nu)$ implies $\nu\in Val^X_H(\exists x_nu)$. In a similar way one can see that $\nu\in Val^X_H(\exists x_nu)$ implies $\mu\in Val^X_H(\exists x_nu)$. Thus the formula $\exists x_n u $ is also correct. Therefore all formulas from $\Phi^0(X)$ are correct. 
So, if $\mu$ and $\nu$ are isomorphic then $LKer^0(\mu)=LKer^0(\nu)$. Using once again the condition of strict logical perfectness we have $\mu\sigma=\nu$ for some $\sigma\in Aut H$ which extends $\alpha: A\to B$. This means that algebra $A$ is homogeneous.

Conversely, let $H$ be homogeneous. Then $\mu\tau\nu$ means that $\mu\sigma=\nu$ for some $\sigma\in Aut H$. Hence, a coset of $\tau$ is an orbit. However, a coset for $\tau$ is defined by a set $T\subset \Phi^0(X)$. This means that the algebra $H$ is strictly logically perfect.

\end{proof}

\begin{remark}
In fact, Theorem \ref{zh} can be deduced from the criterion $\tau=\rho_0$. However, we presented here the original edifying proof.
\end{remark}

\begin{remark}
  Note that from the proof of  Theorem \ref{emb} could be seen that every group is embedded into a homogeneous one. This fact can also be deduced from the method of  HNN-extensions.
\end{remark}

Now we give an example of logically perfect but not a homogeneous group. This example also belongs to G.I. Zhitomirskii.

\begin{proposition}\label{prop:cycl} The infinite cyclic group $\mathbb Z$ is logically perfect.
\end{proposition}
\begin{proof}

Consider the infinite cyclic group which is represented as the additive group of integers $\mathbb Z$. This group is also a $\mathbb Z$-module. Each  subgroup of $\mathbb Z$ is a submodule.  We take the variety of all abelian groups or the variety of $\mathbb Z$-modules as $\Theta$. It is clear that the group $\mathbb Z$ is not homogeneous.  Let us check that $\mathbb Z$ is logically perfect.

Take a set $X=\{x_1,\ldots,x_n\}. $ Pick a variable $y$ and let $Y=\{y,X\}$. Define the map $s:Y\to X$ by the rule $s(x_i)=x_i$, $s(y)=0$. So we have morphisms $s: W(Y)\to W(X)$ and $s_*: \Phi(Y)\to \Phi(X)$. For each formula $v\in \Phi(Y)$ there is the formula $u=s_*v$, $u\in \Phi(X)$ which does not necessarily belong to $\Phi^0(X).$

Let $\mu: W(X)\to \mathbb Z$ be a point. We have $\mu\in Val_{\mathbb Z}^X(u)=s_*Val^Y_{\mathbb Z}(v)$ if and only if $\mu s \in Val^Y_{\mathbb Z}(v)$.

Take two points $\mu$, $\nu: W(X)\to \mathbb Z$ and let $LKer(\mu)=LKer(\nu)$. Let us show that $\mu$ and $\nu$ are conjugated by an automorphism of $\mathbb Z$.
This group has only the identity automorphism and the automorphism which takes an element to the inverse one. We shall check that $\mu=\nu$ or $\mu=-\nu$.

Let $\mu(x_i)=a_i$, $\nu(x_i)=b_i$, $i=1,\ldots, n$. If all $a_i$ are equal to zero, then $\mu=\nu$. So we can assume that $a_1\neq 0$.

For the point $\mu$ we are going to construct a special test formula $v\in \Phi(Y)$ such that $\mu\in Val^X_\mathbb Z(u)$, $u=s_*v$. This will imply $\nu\in Val^X_\mathbb Z (u)$.  First, we shall define $v_0\in \Phi(Y)$. Then $v$ is constructed as $v=\exists y
v_0$.

Let $v_0$ be the formula
$$
(x_1\equiv|a_1|y)\wedge(x_2\equiv sgn(a_2a_1)|a_2|y)\wedge\ldots\wedge (x_n\equiv sgn(a_na_1)|a_n|y).
$$

Here, $|a|$ stands for the absolute value of $a$ and $\it sgn(a)$ is a sign of $a$.

Define a point $\gamma: W(Y)\to \mathbb Z$ by $\gamma(x_i)=\mu(x_i)=a_i$ and $\gamma(y)=sgn(a_1)1$. The point $\gamma$ satisfies $v_0$.  Since $\mu s$ and $\gamma$ coincide on $X$ then the point $\mu s$ satisfies the formula $\exists y v_0\equiv v$. In other words  $\mu\in Val^X_\mathbb Z(u)$, where $u=s_*v$. We have also $\nu s\in Val^Y_\mathbb Z(v)$. From this follows that for some values $c$ of $y$ and  $b_i$ of $x_i$ we have
$$
b_1=|a_1|c,\ldots, b_n=sgn(a_na_1) |a_n|c.
$$
This gives that $b_i=a_ic$ if $a_1>0$,  $b_i=-a_ic$ if $a_1<0$. Interchanging $\mu$ and $\nu$ we have $d$ such that $a_i=b_id$ if $b_1>0$ and $a_i=-b_id$  if $b_1<0$. In such a way we arrive to $c=1$ or $c=-1$. In the first case we have $\mu=\nu$, and in the second one $\mu=-\nu$.
 The proof of the example is finished.
\end{proof}

The group $\mathbb Z$ is not strictly logically perfect since $\mathbb Z$ is not homogeneous. It wood be interesting to get other examples which work in different $\Theta$.

Let us make one more remark. Let us take $\mu:W(X)\to H$ and consider its closure $A=\{\mu\}^{LL}_H=(LKer(\mu))^L_H$. This $A$ is a minimal elementary set and every elementary set in $Hom(W(X),H)$ is the union of such disjoint $A$. If $H$ is logically perfect then these minimal elementary sets coincide with the orbits of $Aut(H)$. This is a property of the lattice of elementary sets in $Hom(W(X),H)$. It is clear that the minimal elementary sets one to one correspond to types.

Now we give two simple examples of  strictly logically perfect abelian groups.

\subsection{Example}
Let us study a concrete example. Let a group $H$ be a discrete direct product
 of all simple cyclic groups of different prime orders.
We show that for a given $H$ all $Aut(H)$-orbits in $Hom(W,H)$ are
elementary sets.

Proceed from the variable $x \in W=W(X)$  and write down a formula
$u=u(x)$ of the form $x\not =1 \wedge x^m=1 \wedge u_0$, where $u_0$
is the conjunction of  all $x^{m_i}\not =1$ by all divisors $m_i$ of
the number  $m$.

Only elements $g$ of the order $m$ satisfy the formula $u=u(x)$.
There is a finite number of such elements, namely $(p_1-1) \ldots
(p_k -1)$, if $m=p_1 \ldots p_k$.

Take a finite $X$ and consider $Hom(W(X),H)$. Let $(\mu(x_1), \ldots
, \mu(x_n))= \bar g = (g^1, \ldots , g^n)$ for a point $\mu : W(X)
\to H$. Represent each $g^i$ as $g^i = g^i_1 \ldots g^i_{k_i}$. The
factors are organized by increasing of their prime orders. The order of
$g^i$ is some $m^i$.

Proceed further from the set of formulas $u(x_1), \ldots  ,u(x_n)$
for the orders $m_1, \ldots , m_n$ and let $u$ be their conjunction.
Pass to the elementary set $Val^X_H(u)=A$. All the points $\mu :
W(X) \to H$ of the type described above are included in this set.

\begin{proposition}\label{prop:orbit}
Every $A= Val^X_H(u)$ is an orbit of the group $Aut(H)$ and every
orbit has such form.
\end{proposition}

\begin{proof}
Let $\mu: W(X) \to H$ be a point in A. For every automorphism
$\sigma \in Aut(H)$ we have the inclusion $\mu \sigma \in A$. We
should check whether $\nu = \mu \sigma$ for any other point $\nu:
W(X) \to H$ from $A$ and some automorphism $\sigma$.

By the definition, we have $\mu \sigma(x) =\mu(x) \circ \sigma =g
\circ \sigma$ for every $x \in X$. Let $g=g_1 \ldots g_k$, where the
order of $g_i$ is $p_i$. Let, besides, $\nu(x) = g'=g'_1 \ldots
g'_k$, where the order of $g'_i$ is also $p_i$ and $g_i$ and $g'_i$
are generators of the cyclic group of the order $p_i$. Therefore the
transition $g_i \to g'_i$ determines automorphisms of these cyclic
subgroups. This holds for each $x \in X$ and, hence, we have an
automorphism $\sigma$, taking $g$ into $g'$. We also have
$\mu \sigma = \nu$.

Prove now that every orbit is some set $A$. Take an arbitrary
$X=\{x_1, \ldots ,x_n\}$ and a point $\mu : W(X) \to H$. A sequence
$\bar g =(g^1, \ldots ,g^n)$ corresponds to this point. A specific
$A$ containing $\mu$ corresponds to $\mu$, and this is valid for
each $A$. Applying to $\bar g$ arbitrary automorphisms $\sigma \in
Aut(H)$ we get the whole orbit $A$.
\end{proof}

There is an infinite number of different orbits, and thus an infinite number of
elementary sets and types. The interesting point here is that every
orbit is finitely defined. It is easy to understand that the group
$H$ satisfies none of the noetherianity conditions and, besides, if
some $H'$ is isotyped with $H$, then $H$ and $H'$ are isomorphic.
Actually this result is derived  from elementary equivalence of  $H$
and $H'$.

Let us prove this fact and then make some remarks on noetherianity.

Let $H$ be a group from the \ref{prop:orbit} and $H'$ a
 group isotyped to $H$. Then they are elementary equivalent and they have the same
identities. Thus, $H'$ is abelian. The formula $\exists x (x^p
\equiv 1)$ in $H$ means that also in $H'$ there exist elements of
the order $p$. The formula $x_1^p \equiv x_2^p \equiv 1 \to
x_1\equiv x_2^{m_1} \vee \ldots \vee x_1\equiv x_2^{m_k}$ with $m_i
< p$ implies that $H'$ contains only one cyclic subgroup of the
order $p$. The formula $x^{p^n} \equiv 1 \to x^p \equiv1$ means that
all $p$-elements are of the order $p$.

All this holds in $H'$. Therefore, $H'$ is a direct product of all
cyclic subgroups of all prime orders, that leads to isomorphism.

Let us show that $H$ is not logically  noetherian. Take an infinite subset $M$
in the set of primes with the infinite complement $M'$. Consider a
formula $\neg (x^p \equiv 1 \to x \equiv 1)$ by all $p \in M$. This
set is not reduced to a finite one since $M'$ is infinite.


\subsection{Additive group of rational numbers}

Proceed now from the additive group  $H$   of rational numbers.

First of all we show that this group is homogeneous. Fix a set $X=\{ x_1, \ldots , x_n\}$. Let $W(X)$
be the free abelian group over $X$. Take the points $\mu, \nu:W(X)\to H$. Let $A$ and $B$ be the images of $\mu$ and $\nu$, respectively, and $\alpha: A\to B$ be an isomorphism. The groups $A$ and $B$ are cyclic and thus $A=\{a\}$, $B=\{b\}$. We assume that $\alpha(a)=b$. Let the sequence $a_1 , \ldots , a_n$, where $a_i=m_ia \in A$ corresponds to $\mu$. Correspondingly, we have $(b_1,\ldots ,b_n)$
for $\nu$. Here, $\alpha(a_i)=b_i=m_i b$, where all $m_i$ are integers. Take $s=b/a$, $sa=b$. Multiplication of the element $h\in H$ by $s$ is an automorphism of the group $H$. This automorphism extends the initial $\alpha$.

If, further, $A$ is the orbit of the point $\mu$ then this orbit is defined by a set of formulas T determined by the kernel of the point $\mu$.

The same orbit one can describe using another approach. Let $x$ be an auxiliary variable and consider the formula
$$u=u(m_1,\ldots, m_n)=\exists x(x^{m_1}=x_1\wedge \ldots \wedge x^{m_n}=x_n).$$
 \noindent
 If we restrict the value of this formula on the initial set $X$ then we come to the orbit we are looking for.

 In the next definition and Theorem \ref{th:locsyc}  a possible use of the auxiliary variable $x$ is taken into account.

\begin{defn}
An algebra $H$ is called logically locally noetherian if for each
finite set $X$ the closure of a finite set of points $\mu_i: W(X)
\to H$ is finitely defined.
\end{defn}

\begin{theorem}\label{th:locsyc}
The locally cyclic group $H$ is logically locally noetherian.
\end{theorem}

\begin{proof}

First of all we show that for every point $\mu: W(X)\to H$ we have $LKer(\mu)= \{u\}^{LL}_H$ for some formula $u$ of the form $u=u(m_1, \ldots ,
m_n)$. Besides, $\{\mu\}^{LL}_H-$ is a finitely definable closure of the point $\mu$.

Let,  $u=u(m_1, \ldots ,
m_n)$. Then $Val_H^{X}(u) =\{u\}^L_H$ and $(Val_H^{X}(u))^L_H
=\{u\}^{LL}_H$. Let now $\mu \in A_0=Val_H^{X}(u)$. Then
$\{\mu\}^L_H \supset {A_0}^L_H =\{u\}^{LL}_H$ and $LKer(\mu) \supset
\{u\}^{LL}_H$. By the definition of the operator $L$ we have
${A_0}^L_H = \bigcap_{\nu \in A_0} LKer(\nu)$. It follows from the
previous considerations that $\mu$ and $\nu$ are isotyped. Hence,
${A_0}^L_H =  LKer(\mu)= \{u\}^{LL}_H$. We have also $
(LKer(\mu))^L_H=A_0$ for arbitrary $\mu \in A_0$.

We will need now a remark on the lattice of all elementary sets in
the given $Hom(W(X),H)$. Let $A$ and $B$ be two such sets,
$A^L_H=T_1$, $B^L_H=T_2$. Let $T^0_1$ be a subset in $T_1$ with
$(T^0_1)^{LL}_H = T_1$. Similarly, take $T^0_2$ in $T_2$ with
$(T^0_2)^{LL}_H = T_2$. Denote by $T^0_1 \vee T^2_0$  a set of all
$u \vee v$, $u \in T^0_1$, $v \in T^0_2$. Check that $(T^0_1 \vee
T^2_0)^L_H = A \cup B$. Indeed, $(T^0_1 \vee T^2_0)^L_H = \bigcap_{u
\vee v} Val^X_H(u \vee v) = \bigcap_{u \vee v}( Val^X_H(u) \cup
Val^X_H(v))= (\bigcap_{u \in T^0_1}Val^X_H(u)) \cup (\bigcap_{v \in
T^0_2}Val^X_H(v)) =  A\cup B$.

Apply this to the group $H$. Let
$u_1= u(m_1, \ldots , m_n)$, $u_2= u({m'}_1, \ldots , {m'}_n)$,
$A=Val^{X}_H(u_1)$, $B=Val^{X}_H(u_2)$. Then $A \cup B =
Val^{X}_H(u_1 \vee u_2)$.  This implies that
the union of any finite number of the sets of the type $A$ is an
elementary finitely defined set.

Let, further, $A_0$ be a finite set of points $\mu: W(X) \to H$.
Every $\mu \in A_0$ belongs to some $Val^{X}_H(u)$, $u=u(m_1, \ldots
, m_n)$. Hence, $A_0$ is contained in a finite union of the sets of
the type $Val^{X}_H(u)$.

Let $A_0$ be a subset in some
$A=(LKer(\mu))^L_H$. Then ${A^L_0}_H= \bigcap_{\nu \in A_0}
LKer(\nu) =LKer(\mu)$ for a point $\mu\in A_0$. The closure is
${A^{LL}_0}_H =A$.

Take now a finite set of points $A_0=\{\mu_1, \ldots , \mu_k\}$,
$\mu_i \in LKer(\mu_i)^L_H$ and show that its closure is finitely
defined. Here we have ${A_0}^L_H = \bigcap_{\mu_i \in A_0}
LKer(\mu_i)$.

It follows from the remarks on elementary sets that ${A^{LL}_0}_H =
\bigcup_{\mu_i} LKer(\mu_i)^L_H = \bigcup_{\mu_i} A_i$, $A_i=
LKer(\mu_i)^L_H$. For the group of rational numbers $H$ all $A_i$
are finitely defined. Then $\bigcup_{\mu_i} A_i$ is finitely defined
as well. Therefore, the closure ${A^{LL}_0}_H$ is finitely defined.

\end{proof}

Now we might compare local logical noetherianity to other noetherianity
conditions. We do not  study this problem here.  Just note that every logically noetherian algebra $H$ is locally logically noetherian.
However, if
every finitely generated subgroup in $H$ is logically noetherian, it
does not mean that $H$ is logically locally noetherian.




\section{Some facts from algebraic logic. Appendix} \label{sec:appendix}

\subsection{Introduction}

For every given variety of algebras $\Theta$ we distinguish an
ordinary pure logic in $\Theta$ and algebraic logic in $\Theta$.
Formulas of algebraic logic are formulas of the pure logic,
compressed by the semantic relation in the given $\Theta$. Polyadic
Halmos algebras and cylindric  Tarski algebras are the main general
structures of the algebraic logic. The essential characteristic
property of these structures is that they admit an infinite set of
variables. Denote this set by $X^0$. (See \cite{H}, \cite{HMT}).

In our case we are forced to
consider a system of all finite subsets $X$ in $X^0$ instead of such
big $X^0$. Denote this system by $\Gamma$. Then we pass to a
multi-sorted algebra with a system of sorts $\Gamma$. We come, in
particular, to multi-sorted Halmos algebras and to Halmos
categories.

\subsection{A category - algebra $Hal_\Theta(H)$}

Let us start with an important example, namely, Halmos category
$Hal_\Theta(H)$. Here $H$ is an algebra in $\Theta$. Objects of this
category are extended boolean algebras $Bool(W(X),H)$. Define
morphisms $s_* : Bool(W(X),H)\to Bool(W(Y),H)$. Denote by $\Theta^0$
the category of free in $\Theta$ algebras $W=W(X)$, $X \in \Gamma^0$.
We have also a category $K^0_\Theta(H)$ of all affine spaces over
$H$. Its morphisms are mappings $\tilde s: Hom(W(Y),H) \to
Hom(W(X),H)$, where $s:W(X) \to W(Y)$ is a morphism in $\Theta^0$
and $\tilde s(\nu) = \nu s: W(Y) \to H$ for $\nu: W(X) \to H$. For
each $A \subset Hom(W(X),H)$ we set $s_* A = \tilde s^{-1} A$. This
determines a morphism in $Hal_\Theta(H)$. Every $s_*$ is also a
homomorphism of Boolean algebras and it is correlated with
quantifiers and equalities (see also Subsection  \ref{subsec:HA}).

Passing to general definitions, let us refine the notion of extended
Boolean algebras.

Recall that in Algebraic Logic (AL) quantifiers are treated as
operations on Boolean algebras. Let $B$ be a Boolean algebra. Its
{\it existential quantifier} is a mapping $\exists : B \to B$ with the
conditions:

1. $\exists 0 = 0$,

2. $\exists a > a$,

3. $\exists(a \wedge \exists b) = \exists a \wedge \exists b $.

\noindent The {\it universal quantifier} $ \forall : B \to B$ is defined
dually:

1. $\forall 1 = 1$,

2. $\forall a < a$,

3. $\forall(a \vee \forall b) = \forall a \vee \forall b $.

\noindent Here $0$ and $1$ are zero and unit of the algebra $B$ and
$a,b$ are arbitrary elements of $B$. The quantifiers $\exists$ and
$\forall$ are coordinated in the usual way: $\overline{\exists
a}=\forall{\overline a}$, $\overline {\forall a}=\exists {\overline
a}$ .

Let $\Theta$ and $W=W(X) \in \Theta$ be fixed and $B$ be a Boolean
algebra. We call $B$ an {\it extended Boolean algebra} in $\Theta$ over
$W(X)$, if

1. There are defined quantifiers $\exists x$ for all $x \in X$ in
$B$ with $\exists x \exists y = \exists y \exists x$ for all $x,y
\in X$.

2. To every formula  $w \equiv w'$, $w,w' \in W$ it corresponds a
constant in  $B$, denoted also by $w \equiv w'$. Here,

2.1. $w \equiv w$ is the unit of the algebra $B$.

2.2. For every $n$-ary operation $\omega\in \Omega$, where $\Omega$
is a signature of the variety $\Theta$, we have
$$
w_1\equiv w'_1 \wedge \ldots \wedge w_n \equiv w'_n < w_1 \ldots w_n
\omega \equiv w'_1 \ldots w'_n \omega.
$$

\noindent We can consider the variety of such algebras for the given
$\Theta$ and $W=W(X)$.

\subsection{Halmos category. A general definition}
\begin{defn} A category $\Upsilon$ is a {\it Halmos category} if: 

1. Every its object has the form $\Upsilon(X)$,  and this object is
an extended Boolean algebra in $\Theta$ over $W(X)$.

2. Morphisms are of the form $s_* : \Upsilon(X) \to \Upsilon(Y)$,
where $s: W(X) \to W(Y)$ are morphisms in $\Theta^0$, $s_*$ are
homomorphisms of Boolean algebras and the transition $s \to s_*$ is
given by a covariant functor $\Theta^0 \to \Upsilon$.

3. There are identities controlling the interaction of morphisms
with quantifiers and equalities. The coordination with the
quantifiers is as follows:

\noindent  $\qquad $ 3.1.
 $s_{1*} \exists x a = s_ {2*} \exists x
a, \quad a \in \Upsilon(X)$, if $s_1 y = s_2y$ for every $y
 \in X, \; y \neq
x$.

\noindent $\qquad $ 3.2. $s_{*}\exists x a = \exists (sx) (s_*a) $
if $ sx = y \in Y$ and $y = sx$ is not in the support of $sx'$, $x'
\in X, \; x' \neq x$.

4.  The following conditions describe coordination with equalities:

\noindent
 $\qquad $ 4.1. $s_* (w\equiv w') = (sw \equiv sw')$ for
$s\colon W(X) \to W(Y)$, $w, w' \in W(X)$.

\noindent $\qquad $ 4.2. $s^x_w a \wedge (w \equiv w') < s^x_{w'} a$
for an arbitrary $a \in \Upsilon(X), x \in X,$ where $w,w' \in
 W(X)$, and $s^x_w\colon W(X) \to W(X)$ is defined by the rule:
$ s^x_w (x) = w, s^x_w(y) = y, y \in X,\; \;  y \neq x$.

\noindent
\end{defn}
The category $Hal_\Theta(H)$ is an example of the Halmos category.
Another important example is the category of formulas
$Hal^0_\Theta$ of the algebras of formulas $Hal^0_\Theta(X) =
\Phi(X)$. {\it This category plays in logical
geometry  the same role as the category $\Theta^0$ plays in AG.}
\subsection{Halmos algebras}\label{subsec:HA}
We deal with multi-sorted Halmos algebras, associated with Halmos
categories. Describe first the signature  $L_X$  . Take $L_X = \{
\vee, \wedge, ^-, \exists x, x \in X, M_X \}$ for every $X$. Here
$M_X$ is a set of all equalities over the algebra $W=W(X)$. We  add
all $s: W(X) \to W(Y)$ to all $L_X$, treating them as symbols of unary
operations. Denote the new signature by $L_\Theta$.

Consider further algebras $\Upsilon = (\Upsilon_X, X \in \Gamma^0)$.
Every $\Upsilon_X$ is an algebra in the
 signature $L_X$ and an unary operation (mapping) $s_*: \Upsilon_X \to \Upsilon_Y$  corresponds to every
$s: W(X) \to W(Y)$.

\begin{defn}
We call an algebra $\Upsilon$ in the signature
$L_\Theta$ {\it a Halmos algebra,} if

1. Every $\Upsilon_X$ is an extended Boolean algebra in the
signature $L_X$.

2. Every mapping $s_*: \Upsilon_X \to \Upsilon_Y$  is a homomorphism of Boolean
algebras.

3. The identities, controlling interaction of operations $s_*$ with
quantifiers and equalities are the same as in the definition of
Halmos categories.


\noindent

\end{defn}
 It is clear now that each Halmos category $\Upsilon$ can
be viewed as a Halmos algebra and vice versa.
 In particular, this relates to $Hal_\Theta(H)$. Recall also that homomorphisms of multi-sorted algebras work componentwise.

\subsection{Categories and algebras of formulas}
Denote by $M=(M_X, X \in \Gamma)$ a multi-sorted set with the
components $M_X$.

Take the {\it absolutely free algebra} $\Upsilon^0 = (\Upsilon^0_X, X \in
\Gamma)$ over $M$ in the signature $L_\Theta$. Elements of each
$\Upsilon^0_X$ are First Order Logic (FOL) formulas  which are
inductively constructed  from the equalities using the signature
$L_\Theta$. So, $\Upsilon^0$ is a multi-sorted algebra of pure FOL
formulas over equalities.

Denote by $Hal_\Theta$ the variety of $\Gamma$-sorted Halmos algebras
in the signature $L_\Theta$. 
Denote by $Hal^0_\Theta$ the free algebra of this variety over the
multi-sorted set of equalities $M=(M_X, X\in \Gamma)$.

The same $M$ determines the
homomorphism $\pi = (\pi_X, X \in \Gamma) : \Upsilon^0 \to
Hal^0_\Theta$. If $u \in \Upsilon^0_X$, then the image $u^{\pi_X} =
\bar u$ in $Hal^0_\Theta (X)$ is viewed as a compressed formula.

{\it Setting $Hal^0_\Theta(X) = \Phi(X)=(\Upsilon^0_X)^{\pi_X}$ we get the wanted algebra of
compressed formulas which we used throughout the paper}. This  is an extended Boolean algebra with additional operations of type $s_*$.

Recall that the Halmos algebra of formulas $Hal^0_\Theta$ is also a
Halmos category.
 We have a covariant functor
$\Theta^0 \to Hal^0_\Theta$.

\subsection{Value of a formula}
The value $Val^X_H(w\equiv w')$ corresponds to each equality
$w\equiv w'$, $w, w' \in  W(X)$.
 This determines a mapping
$Val_H: M \to Hal_\Theta (H)$ which is uniquely extended up to
homomorphisms $$Val^0_H : \Upsilon^0 \to Hal_\Theta (H),$$ and
$$Val_H: Hal^0_\Theta \to Hal_\Theta (H).$$
\noindent
 For every $X \in
\Gamma$ we have a commutative diagram


$$
\CD
\Upsilon^0_X @[2]> Val^{0X}_H  >> Bool(W(X),H)\\
 @[2]/SE/\pi_X //@.@.\;    @/NE//Val^X_H/\\
 @. \Phi(X) \\
\endCD
$$

\noindent Thus, for every $u \in \Upsilon^0_X$ and the corresponding
$\bar u \in \Phi(X)$ we have the values $Val^{0X}_H (u) =
Val^{X}_H (\bar u)$.

Let us make a remark on the kernel of the homomorphism $Val_H$. We
have
$$\Ker(Val_H) = Th(H) = (Th^X (H), X \in \Gamma).$$
Here $Th(H)= (Th^X (H), X \in \Gamma),$ is the {\it elementary theory} of
the algebra $H$, i.e., 
the set of formulas $u \in Th^X (H)$ such that $Val^X_H(u)=
\Hom(W(X),H)$. It is clear also that the image $ImVal_H$ is a
subalgebra in $Hal_\Theta(H)$  which consists of one-defined
elementary sets.

\subsection{The main theorem}
\begin{theorem}\label{th:main}
\cite{P7} The variety $Hal_\Theta$ is generated by all algebras $Hal_\Theta(H)$, $H \in\Theta$.
\end{theorem}

This means that identities of all $Hal_\Theta(H)$ determine the variety of Halmos algebras
$Hal_\Theta$.

If $\Theta_1$ is a subvariety in $\Theta$, then the variety $Hal_{\Theta_1}$ in $Hal_\Theta$,
generated by all $Hal_{\Theta_1}(H)$, $H \in \Theta_1$, corresponds to $\Theta_1$. Therefore,
if $H_1$ and $H_2$ are algebras from $\Theta_1$, then they are isotyped in $\Theta_1$ if and
only if they are isotyped in $\Theta$.

Note also the following general observation. It is clear that there is a canonical homomorphism of multi-sorted algebras
$$
\pi_H: \Upsilon^0\to Hal_\Theta(H),
$$
\noindent
where $\pi_H=(\pi_H^X, x\in  \Gamma), $ and $\pi_H^X$ are homomorphisms  $\pi_H^X:\Upsilon^0_X\to Bool (W(X),H)).$   This homomorphism is unique because it takes equalities to the corresponding equalities. Hence the kernel $Ker(\pi_H)$ is also the congruence of identities of the algebra $Hal_\Theta(H)$. Consider $\widetilde{\pi}=\bigcap_HKer(\pi_H)$. This is the congruence of all identities of the variety generated by all $Hal_\Theta(H)$.

Theorem \ref {th:main} now means that there is the equality $$Hal_\Theta^0=\Upsilon^0/\widetilde{\pi};\ \  \Phi(X)=\Upsilon^0_X/\widetilde{\pi}^X,$$ \noindent
where $\widetilde{\pi}^X=\pi_X. $ It can be shown (see \cite{P9}, \cite{P7}) that the congruence $\widetilde{\pi}$ corresponds to Lindenbaum-Tarski congruence.

We gave the necessary information from algebraic logic. The conditions from Subsection 1.2 are now realized.

\subsection{Category of elementary sets}
First consider the category $Set_\Theta (H)$ of affine sets
over an algebra $H$. Its objects are of the form $(X,A)$, where $A$
is an arbitrary subset in the affine space $\Hom(W(X),H)$. The
morphisms are

$$
[s] : (X,A) \to (Y,B).
$$
Here $s:W(Y) \to W(X)$ is a morphism in $\Theta^0$. The
corresponding $\tilde s: \Hom(W(X),H) \to \Hom(W(Y),H)$ should be
coordinated with $A$ and $B$ by the condition: if $\nu \in A\subset
\Hom(W(X),H)$, then $\tilde s (\nu) \in B\subset \Hom(W(Y),H)$. Then
the induced mapping $[s]:A \to B$ we consider as a morphism $(X,A)
\to (Y,B)$.

Now we define the category of algebraic sets $K_\Theta(H)$ and the
category of elementary sets $LK_\Theta(H)$. Both these categories
are full subcategories in $Set_\Theta (H)$ and  are viewed as
important invariants of the algebra $H$. We call them AG- and
LG-invariants of $H$.

The objects of  $K_\Theta(H)$ are of the form $(X,A)$, where $A$ is
an algebraic set in $\Hom(W(X),H)$. If we take for $A$ the
elementary sets, then we are getting  the category of elementary
sets $LK_\Theta(H)$. The category $K_\Theta(H)$ is a full
subcategory in $LK_\Theta(H)$.

As it was mentioned, if algebras $H_1$ and $H_2$ are isotyped, then
the categories $LK_\Theta(H_1)$ and $LK_\Theta(H_2)$ are isomorphic.

\subsection{Model Theory and algebraic logic}

Let us recall some known facts. Along with an algebra of formulas
$\Phi = \Phi(X)$ we consider the algebra of pure formulas 
$\Upsilon^0_X$. Here formulas are identified with their records. Variables
from the set $X$ take part in the records. A variable may be either
bound by a quantifier, or  free. A formula $u$ is called closed (or a
proposition) if all its variables are bound. If $u$ is a closed
formula, then its value $Val^X_H(u)$ is a unit or zero of the
algebra $Bool(W(X),H)$ (the whole space $Hom(W(X),H)$ or the empty
set of points). The set of closed formulas $T$ in $\Upsilon^0$ or in
$Hal^0_\Theta$ is called a theory. A theory is satisfiable if it has
a model. A theory $T$ is called complete if every closed formula $u$
belongs to $T$ or $\neg u$ belongs to $T$. The theory $Th(H)$ is
always complete. A theory $T$ is called categorical in the given
cardinal $\alpha$ if any two of its models of the cardinality $\alpha$
are isomorphic.

This paper is one in the series of papers related to universal
algebraic geometry \cite{KMR}, \cite{BMR},\cite{MR},\cite{P6},\cite{P2},\cite{P4},\cite{P8},etc.
As we have mentioned, in the frameworks of the
considered theory there arise various problems close to algebra and
Model Theory. They seem to be new, although some of them may look
simple to specialists (I mean specialists in Model Theory).

The following remarks will concern pure logic and algebraic logic in
Model Theory. Model Theory is a combination of
syntax and semantics. Syntax (languages) play an essential role
comparable to that of semantics (models). One can speak of syntactic
structure related to languages and of theories in languages. All
this is applicable to concrete mathematical problems.

Algebraic logic is not just a syntax, but nevertheless it works
in Model Theory. It is helpful to treat homomorphisms $Val^X_H
:\Phi(X) \to Bool(W(X),H)$ interrelating formula and its value.
Elementary theory can be presented as a kernel of such homomorphism.
In concern to theory of types one can speak of logical kernel of a
point. This kernel $LKer(\mu)$ automatically turns to be ultrafilter
in the algebra of formulas $\Phi =\Phi(X)$. There are a lot of other
reasons to apply algebraic logic in Model Theory.

Let us once more name open problems:

\begin{problem}
Consider infinite logically noetherian algebras in different varieties
$\Theta$.
\end{problem}

\begin{problem}
Consider separable algebras in different varieties $\Theta$.
\end{problem}

\begin{problem}
What are the varieties $\Theta$ such that free algebras in $\Theta$ are logically separable.
\end{problem}

\begin{problem}
Construct automorphically finitary algebras $H$ in different varieties $\Theta$. In other words, we want to get algebras $H$ in $\Theta$ such that   for every finite $X$ there is a finite number
of $Aut(H)$-orbits in the space $Hom(W(X),H)$.
\end{problem}

\begin{problem}
A general problem: types and isotypeness in multi-sorted algebras.
Consider this problem for the variety of representations of
groups $Rep-K$.
\end{problem}

Let us note that in \cite{PZ} there was  Problem 1.23  concerning
non-isomorphic isotyped  abelian groups. We see now that this problem can be
easily solved (Proposition \ref{prop:inf_dim}). In fact, the following problem is actual for abelian
groups:

\begin{problem}
Find conditions when two abelian groups are elementary equivalent but not isotyped.
\end{problem}

Finally,

\begin{problem}
Which noetherian groups are not logically noetherian.
\end{problem}

In all the cases we mean special  $\Theta$ and $H\in
\Theta$. For instance, $\Theta$ may be modules, vector spaces, linear
algebras, Lie, associative algebras (finitely dimension or not), etc.
The concrete always  illuminates the general.

\subsection{Acknowledgements}
The author is pleased to thank Zlil Sela who noticed that the logical kernel of a point is a type,
and Alexei Miasnikov who payed my attention on the fact that the logic-geometrical equivalence of algebras is
the same as isotypeness of algebras. I am also grateful to my colleagues Yu.Ershov, E.Katsov, V.Remeslennikov, E.Rips, A.Olshansky, G.Zhitomirski,  B.Zilber
 and others for the constant support.

The paper was written in Jurmala, Latvia. I had perfect working conditions thanks to my
good friends Anna Efimenko and Dima Koval, and Alla and Igor Duman. The manuscript was typed and prepared
by E.\& T.Plotkin.


\end{document}

As it was mentioned above, the category $\Hal^0_\Theta$ plays in
logical geometry the same role as the category $\Theta^0$ does in
universal algebraic geometry.

\bibitem{BM1} G.Baumslag, A.~Myasnikov, V.~Remeslennikov, Algebraic geometry
over groups I,  J. of Algebra, {\bf 219:1} (1999) 16 -- 79.

\bibitem{MR2} A.~Myasnikov, V.~Remeslennikov, Algebraic geometry
over groups II, Logical foundations J. of Algebra, {\bf 234:1}
(2000) 225 -- 276.

\bibitem{Pl1} B.~Plotkin,  Algebraic logic, varieties of
algebras and algebraic varieties, in Proc. Int. Alg. Conf., St.
Petersburg, 1995,  St.Petersburg,  1999, p. 189 -- 271.

\bibitem{Pl2} B.~Plotkin, Varieties of algebras and algebraic
varieties,  Israel Math. Journal, {\bf 96:2} (1996),  511 -- 522.

\bibitem{Pl3} B.~Plotkin, Varieties of algebras and algebraic varieties.
Categories of  algebraic varieties. Siberian Advanced Mathematics,
       Allerton Press, {\bf 7:2} (1997), 64 -- 97.

\bibitem{Pl4} B.~Plotkin, Some notions of algebraic geometry in
universal algebra,  Algebra and Analysis, {\bf 9:4} (1997), 224 --
248, St.Peterburg Math. J., {\bf 9:4}, (1998)  859 -- 879.


\bibitem{Pl5} B.~Plotkin, Zero divisors in group-based algebras.
Algebras without zero divisors, Bul. Acad. Sci. Mold. Mat., {\bf 2},
(1999), 67 -- 84.

\bibitem{Pl6} B.~Plotkin, Seven lectures on the universal algebraic geometry,
Preprint,(2002),  Arxiv:math, GM/0204245, 87pp.

\bibitem{Pl7} B.~Plotkin, Infinitary quasi-identities and infinitary
quasivarieties, Proc. Latvian Acad. Sci., Section B, {\bf 57} N.3/4,
(2003)  111 -- 112.

\bibitem{Pl8} B.~Plotkin,  {\it  Problems in algebra inspired by universal algebraic
geometry,} Fundamental and Applied mathematics, {\bf 10:3} (2004),
p. 181 -— 197, and { http://arxiv.org/ abs/ math. GM/0406101},
(2004) 21 pp.



\bibitem{Pl9} B.~Plotkin,  {\it Algebraic geometry in the variety $Mod-K$},
Manuscript.

\bibitem{Pl10} B.~Plotkin, {\it Algebras with the same algebraic geometry}, Proceedings of
the International Conference on Mathematical
 Logic, Algebra and Set Theory,  dedicated to 100 anniversary
of P.S.Novikov, Proceedings of the Steklov Institute of Mathematics,
MIAN, {\bf v.242}, (2003), 176 -- 207, and { http://arxiv. org/
math. math.GM/0210194}.


\bibitem{Pl12} B.~Plotkin,{\it Geometrical equivalence, geometrical
similarity, and geometrical compatibility of algebras}, Zapiski
Nauch. Sem. POMI, {\bf 330} (2006), 201 -- 222.

\bibitem{Pl13} B.~Plotkin, {\it Algebras with the same logic},
Manuscript.

\bibitem{PPT} B.~Plotkin, E.~Plotkin, A.~Tsurkov, "Geometrical equivalence of
groups", Communications in Algebra, {\bf 27:8} (1999), 4015 -- 4025.

\bibitem{PlTs} B.~Plotkin, A.~Tsurkov, Action type geometrical
equivalence of representations of groups, Algebra and Discrete
Mathematics, {\bf 4},  (2005), 48 –- 79, see also
http://arxiv.org/abs/math.RT/0501337, (2004).

\bibitem{PV} B.~Plotkin,  S.~Vovsi, {\it Varieties of group representations. General Theory, connections and
applications}, Riga "Zinatne", 1983 (in Russian).

\bibitem{PlZh1} B.~Plotkin, G.~Zhitomirski, {\it On automorphisms of categories of universal algebras},
International Journal of Algebra and Computations,
 (2007), see Arxiv: math.CT/0411408, 2004.

\bibitem{PlZh2} B.~Plotkin, G.~Zhitomirski, {\it On automorphisms of
categories of free algebras of some varieties}, Journal of Algebra,
{\bf 306:2}, (2006), 344 -- 367. 

\end{thebibliography}

\end{document}